\documentclass[pdflatex,sn-mathphys-num]{sn-jnl}

\usepackage{amsmath,amssymb,amsfonts}
\usepackage{amsthm}
\usepackage{mathtools}
\usepackage{booktabs}
\usepackage{multirow}
\usepackage{tabularx}
\usepackage{array}
\usepackage{xcolor}
\usepackage{graphicx}
\usepackage{algorithm}
\usepackage[noend]{algpseudocode}
\usepackage{enumitem}
\usepackage{tikz}
\usepackage{graphicx}
\usetikzlibrary{arrows.meta,positioning,fit}
\usepackage{comment}

\theoremstyle{thmstyleone}
\newtheorem{theorem}{Theorem}
\newtheorem{proposition}[theorem]{Proposition}
\theoremstyle{thmstyletwo}

\theoremstyle{thmstylethree}

\begin{document}

\title[Multi-Channel Geopolitical Supply Chain Reconfiguration]{Robust global supply chain reconfiguration under tariff and geopolitical uncertainty}

\author*[1]{\fnm{Ashish} \sur{Chandra}}\email{achand6@ilstu.edu}
\affil*[1]{\orgdiv{College of Business}, \orgname{Illinois State University}, \orgaddress{\city{Normal}, \postcode{61761}, \state{IL}, \country{USA}}}

\abstract{Global supply chains are increasingly exposed to trade policy shocks
and physical disruptions arising from common geopolitical events. We develop
a robust optimization framework with two decision stages for strategic supply
chain reconfiguration. Supplier qualification and transportation arc
activation decisions are made before uncertainty is realized, while
operational flows adapt afterward. Geopolitical uncertainty is represented
through country level shocks that can jointly trigger tariff escalation,
policy induced restrictions on cross border capacity, facility disruptions,
and physical transportation failures. The resulting min max min problem is
reformulated by dualizing the operational recourse problem, yielding a
nonconvex mixed integer bilinear adversarial problem solved within a column
and constraint generation procedure. Comparison with complete enumeration
verifies the implementation on small instances, while experiments on 60
synthetic supply chain networks of varying sizes assess computational
performance. The results show that omitting geopolitical transmission
channels during strategic design can lead to substantial deterioration in
cost and service under full geopolitical stress. Broader protection is
achieved primarily through changes in transportation network configuration
rather than increased supplier qualification. Robust performance also
depends on the severity and composition of geopolitical consequences, with
restrictions on cross border capacity exerting a stronger effect than
comparable proportional changes in tariff severity over the tested ranges.
A geographically grounded case analysis further shows that the value of
individual uncertainty channels depends on network structure and geographic
exposure.}

\keywords{robust optimization; supply chain resilience; geopolitical risk;
trade policy uncertainty; network reconfiguration; column and constraint generation}

\maketitle

\section{Introduction}\label{sec:intro}

Global supply chains are increasingly exposed to shifts in international
trade policy that can alter sourcing costs, manufacturing footprints,
transportation routes, and market access. Tariff escalation, retaliatory
measures, sanctions, export controls, and border restrictions can require
firms to reconsider strategic sourcing and logistics decisions made under
substantially different trade conditions. At the same time, the geopolitical
events underlying such policy changes may also disrupt facilities or
transportation connections. Consequently, strategic supply chain
reconfiguration requires consideration not only of changes in the cost of
cross border trade, but also of the multiple operational consequences through
which the same geopolitical event may affect the network.

Supply chain network design under uncertainty has received extensive
attention in the operations research literature. Existing studies
employ stochastic, robust, and related optimization approaches to
account for uncertain demand, supply, capacity, and disruption risk
\cite{Govindan2017,Aldrighetti2021}. Robust optimization has also
been applied to integrated network configuration and operational
decisions, including facility location, inventory management, and
transportation routing \cite{ShangZhangJiaAlmanaseer2022, Goli_2023}.
Within the resilience literature, facility failures, supply and
demand interruptions, and investments in protection or redundancy
have been incorporated into strategic network design models
\cite{Jabbarzadeh2016,HuTangTian2025,LuoLiZhao2025}. Lotfi et al.~\cite{LotfiShoushtariAli2025} develop a
bilevel supply chain network design model incorporating
risk, robustness, and resilience strategies.
Yildiz et al.~\cite{YildizYoonTalluriHo2016} examine the tradeoff
between network cost and reliability, while Liu et
al.~\cite{LiuXiWang2023} investigate supply chain resilience
strategies under budget constraints. More recently, Yildiz et
al.~\cite{YildizYanEtAl2025} examine the efficiency of combining
multiple risk mitigation strategies under complex disruption
patterns. Dolgui et al.~\cite{DolguiIvanovSimchiLevi2025}
emphasize the role of supply chain stress testing in assessing
resilience and viability under disruptions. Together, these
studies highlight the importance of network configuration,
operational flexibility, and systematic evaluation of
resilience strategies under adverse conditions.

A related literature examines international sourcing
decisions and the effects of trade policy interventions
on global supply chains. Kandil et
al.~\cite{KandilBattaiaHammami2020} review analytical
models for sourcing and location decisions in the context
of globalization and changing international trade
conditions. Dong and Kouvelis \cite{DongKouvelis2020}
discuss the implications of tariffs for global supply
chain network configuration, while Grossman et
al.~\cite{GrossmanHelpmanRedding2024} examine how
unexpected tariffs affect supplier relationships and
sourcing decisions. Nagurney et
al.~\cite{NagurneyBesikNagurney2019} analyze tariff rate
quotas in competitive global supply chain networks, while
Feng et al.~\cite{FengEtAl2022} investigate the effects
of trade policies on supply chain network equilibrium.
Nagurney and Samadi \cite{NagurneySamadi2025} extend
the equilibrium perspective by incorporating ad valorem
tariffs and labor constraints.

Recent research also emphasizes the broader consequences
of tariff uncertainty and geopolitical disruptions.
Ivanov and Dolgui \cite{IvanovDolgui2025} discuss the
ripple effects and deep uncertainty associated with
tariff shocks, while Ivanov \cite{Ivanov2026} develops
a decision support framework for responding to tariff
induced uncertainty and disruptions. Cruz \cite{Cruz2026}
examines tariff uncertainty and risk aversion in a
multiperiod supply chain equilibrium model. Particularly
related to the present study, Ahmadi Digehsara et
al.~\cite{AhmadiDigehsara2025} develop a robust
optimization framework with two decision stages for
geopolitical disruptions in which strategic alliances
influence exposure to tariff increases and quota based
trade restrictions.

These streams provide sophisticated treatments of physical disruption risk
and trade policy interventions, but an important modeling issue arises when
different consequences originate from a common geopolitical event. Changes
in trade costs or capacity and disruptions to facilities or transportation
need not constitute independent sources of uncertainty. They may instead
represent alternative or simultaneous consequences of the same country level
shock. Treating these consequences as independent can admit combinations that
do not reflect their geographic origin, whereas representing only one
consequence channel can omit other ways in which the same geopolitical event
affects the network. This motivates an uncertainty architecture that
separates the geographic origin of geopolitical uncertainty from its
consequences for the supply chain and explicitly controls which combinations
of consequences may occur jointly.

To address this issue, we develop a robust optimization model with two
decision stages for strategic global supply chain reconfiguration under
geopolitical uncertainty with multiple consequence channels. Before
uncertainty is realized, the firm qualifies suppliers and activates
transportation arcs, thereby determining the sourcing and routing alternatives
available to the network. After the geopolitical realization is observed,
operational flows adapt over the resulting network. Country level shocks
determine the geographic exposure of network components, while separate
consequence channels represent their effects on trade conditions, facility
availability, and transportation connectivity. Budgets for individual
channels regulate the extent of each type of consequence, while a joint
budget controls their simultaneous occurrence.

The paper makes three main contributions. First, we develop an uncertainty architecture in which geopolitical shocks
originate at the country level and may propagate through multiple trade
policy and physical disruption channels. Restricted versions of the
uncertainty set are obtained by selectively excluding consequence channels.
This common structure allows us to compare strategic designs that consider
only tariff escalation, tariff escalation together with restrictions on cross
border capacity, or physical disruptions to facilities and transportation
with designs that consider all four consequence channels. It therefore
provides a systematic basis for assessing the implications of omitting
particular transmission channels from strategic network design.
Second, we integrate supplier qualification and transportation arc activation
with adaptive operational recourse in a unified robust model for supply chain
reconfiguration. Supplier qualification creates alternative sourcing options,
while transportation arc activation determines the routing alternatives
available after uncertainty is realized. The model therefore captures two
complementary mechanisms for strategic resilience and their interaction with
adaptive operational decisions.

Third, we develop a column and constraint generation solution approach for
the resulting robust optimization problem and conduct a systematic
computational evaluation of the proposed framework. Comparison with complete
enumeration verifies the implementation on small instances. We then compare
strategic configurations developed under different uncertainty assumptions
by exposing each configuration to the same full range of geopolitical
consequences. This analysis quantifies the implications of omitting particular
transmission channels during strategic design. Additional experiments examine
how the results vary with the structure and severity of geopolitical
uncertainty and with the geographic characteristics of the network.

The remainder of the paper is organized as follows.
Section~\ref{sec:problem} describes the decision setting and develops the
structure of geopolitical uncertainty. Section~\ref{sec:model} presents the
robust optimization model and defines the restricted uncertainty
environments. Section~\ref{sec:reformulation} develops the reformulation and
solution procedure. Section~\ref{sec:computational} presents the computational
study and discusses the resulting insights. Section~\ref{sec:conclusion}
concludes the paper.

% TODO: Insert the final verified literature citations and detailed positioning before submission.

%%%%%%%%%%%%%%%%%%%%%%%%%%%%%%%%%%%%%%%%%%%%%%%%%%%%%%%%%%%%%%%%%%%%%%%%%%%%%%%%
\section{Problem setting and geopolitical uncertainty}
\label{sec:problem}
%%%%%%%%%%%%%%%%%%%%%%%%%%%%%%%%%%%%%%%%%%%%%%%%%%%%%%%%%%%%%%%%%%%%%%%%%%%%%%%%

We consider a firm that strategically configures a global supply chain in the
presence of geopolitical uncertainty. The supply chain consists of four
echelons: suppliers, manufacturing plants, warehouses, and retailers located
across multiple countries. Material flows sequentially from suppliers to
plants, from plants to warehouses, and from warehouses to retailers. Because
facilities and transportation connections span national boundaries, the
network is exposed to changes in trade conditions as well as physical
disruptions associated with geopolitical events.

The firm selects its strategic network configuration before geopolitical
uncertainty is realized, incurring costs for supplier qualification and
transportation arc activation. Once geopolitical conditions are observed,
material flows are adjusted within the available network capacities.
Demand that cannot be served because of capacity limitations, trade
restrictions, or physical disruptions is represented as unmet demand and
incurs a penalty cost. The objective is to determine a strategic
configuration that minimizes the sum of the initial reconfiguration cost
and the worst case operating cost under admissible geopolitical
realizations.

%The firm makes decisions in two stages. Before geopolitical uncertainty is realized, it determines the strategic configuration of the network by qualifying suppliers and activating transportation arcs. These decisions determine the sourcing and routing alternatives available during subsequent operations and incur qualification and activation costs. After geopolitical conditions are realized, the firm adapts material flows over the available network. Demand that cannot be served because of capacity limitations, trade restrictions, or physical disruptions is represented as unmet demand and incurs a penalty cost. The objective is to select a strategic configuration that minimizes the sum of the initial reconfiguration cost and the worst case operating cost after uncertainty is realized.

The defining feature of the uncertainty structure is the distinction between
the geographic origin of a geopolitical event and its consequences for the
supply chain. A country level shock represents an adverse geopolitical event,
such as a trade dispute, economic sanction, political instability, or
regional conflict, associated with a particular country. Such a shock creates
exposure for facilities located in that country and for transportation arcs
connected to it. It does not require every exposed component to experience an
adverse consequence. Instead, the country shock establishes which network
components are eligible to be affected, while separate consequence variables
and uncertainty budgets determine which adverse effects are realized.

Four consequence channels are considered. The trade cost channel represents
tariff escalation on international transportation arcs. The trade capacity
channel represents reductions in usable international transportation capacity
resulting from measures such as quotas, sanctions, export controls, or border
restrictions. The facility channel represents disruptions to suppliers,
manufacturing plants, and warehouses. The transportation channel represents
physical failures of transportation arcs. These channels distinguish changes
in the economic and regulatory conditions of international trade from losses
in the physical availability of facilities and transportation connections.

Dependence among these consequences arises from their common geographic
exposure. A facility consequence can occur only when the country hosting that
facility is shocked, while a consequence affecting a transportation arc can
occur only when at least one of its endpoint countries is shocked. Budgets
for individual consequence channels limit the extent of each type of
disruption or trade intervention. A joint budget additionally limits the
composition of consequences that may occur simultaneously across channels.
The resulting construction retains the flexibility of budgeted robust
optimization while preventing consequences from being assigned to network
components that have no exposure to a selected country shock.

%
% ------------------------------------------------------------
% Conceptual framework figure
% ------------------------------------------------------------
% ------------------------------------------------------------
% Figure
% ------------------------------------------------------------
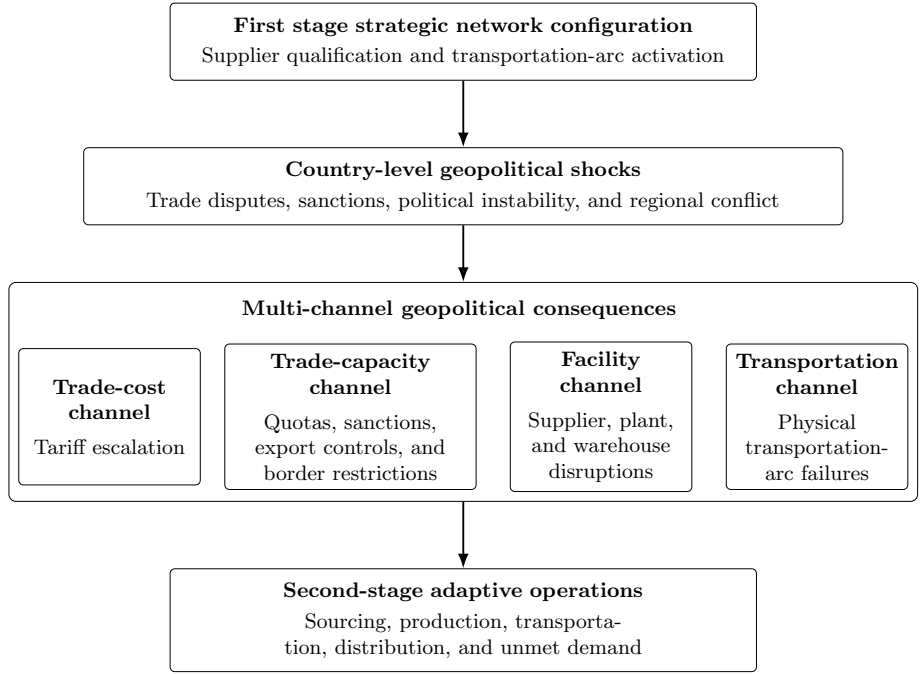
\begin{figure}[htbp]
\centering

\resizebox{0.92\linewidth}{!}{%
\begin{tikzpicture}[
    font=\small,
    stage/.style={
        draw,
        rounded corners=2pt,
        align=center,
        text width=8.4cm,
        minimum height=1.0cm,
        inner sep=6pt
    },
    channel/.style={
        draw,
        rounded corners=2pt,
        align=center,
        text width=2.45cm,
        minimum height=2.05cm,
        inner sep=4pt
    },
    outer/.style={
        draw,
        rounded corners=3pt,
        inner sep=4pt
    },
    arrow/.style={
        -{Latex[length=2.2mm]},
        thick
    }
]

% ------------------------------------------------------------
% Stage 1
% ------------------------------------------------------------

\node[stage] (stage1) {
\textbf{First stage strategic network configuration}\\[2pt]
Supplier qualification and transportation-arc activation
};

% ------------------------------------------------------------
% Country-level geopolitical shocks
% ------------------------------------------------------------

\node[
    stage,
    below=1.0cm of stage1,
    text width=11.0cm,
    minimum height=0.85cm
] (shock) {
\textbf{Country-level geopolitical shocks}\\[2pt]
Trade disputes, sanctions, political instability, and regional conflict
};

\draw[arrow] (stage1.south) -- (shock.north);

% ------------------------------------------------------------
% Title inside the grouped box
% ------------------------------------------------------------

\node[
    below=1.0cm of shock,
    font=\small\bfseries,
    align=center
] (multititle) {
Multi-channel geopolitical consequences
};

% ------------------------------------------------------------
% Invisible center anchor for the four-box row
% ------------------------------------------------------------

\coordinate (rowcenter) at ([yshift=-1.35cm]multititle.south);

% ------------------------------------------------------------
% Four channels in ONE centered horizontal row
% ------------------------------------------------------------

\node[
    channel,
    anchor=east
] (tradecost) at ([xshift=-3.95cm]rowcenter) {
\textbf{Trade-cost channel}\\[4pt]
Tariff escalation
};

\node[
    channel,
    anchor=east,
    text width=3.5cm,
    minimum height=1.45cm
] (tradecap) at ([xshift=0.2cm]rowcenter) {
\textbf{Trade-capacity channel}\\[4pt]
Quotas, sanctions, export controls, and border restrictions
};

\node[
    channel,
    anchor=west
] (facility) at ([xshift=0.7cm]rowcenter) {
\textbf{Facility channel}\\[4pt]
Supplier, plant, and warehouse disruptions
};

\node[
    channel,
    anchor=west
] (transport) at ([xshift=3.95cm]rowcenter) {
\textbf{Transportation channel}\\[4pt]
Physical transportation-arc failures
};

% ------------------------------------------------------------
% Outer grouped box
% ------------------------------------------------------------

\node[
    outer,
    fit=(multititle)(tradecost)(tradecap)(facility)(transport)
] (channelblock) {};

% ------------------------------------------------------------
% Straight vertical arrow into grouped box
% ------------------------------------------------------------

\draw[arrow] (shock.south) -- (channelblock.north);

% ------------------------------------------------------------
% Stage 2
% ------------------------------------------------------------

\node[
    stage,
    below=1.0cm of channelblock
] (stage2) {
\textbf{Second-stage adaptive operations}\\[2pt]
Sourcing, production, transportation, distribution, and unmet demand
};

% ------------------------------------------------------------
% Straight vertical arrow out of grouped box
% ------------------------------------------------------------

\draw[arrow] (channelblock.south) -- (stage2.north);

\end{tikzpicture}%
}

\caption{Decision sequence \& propagation of geopolitical uncertainty
in supply chain reconfiguration.}

\label{fig:framework}

\end{figure}

Figure~\ref{fig:framework} summarizes the decision sequence and the
propagation of geopolitical uncertainty. The firm first selects its strategic
network configuration. An admissible geopolitical realization then determines
the affected countries and the consequences that occur across the four
channels, subject to the geographic coupling conditions and uncertainty
budgets. After observing the realization, the firm adapts its operational
flows over the resulting network. Section~\ref{sec:model} formalizes this
decision structure and the associated uncertainty set.

%%%%%%%%%%%%%%%%%%%%%%%%%%%%%%%%%%%%%%%%%%%%%%%%%%%%%%%%%%%%%%%%%%%%%%%%%%%%%%%%
\section{Robust optimization model}
\label{sec:model}
%%%%%%%%%%%%%%%%%%%%%%%%%%%%%%%%%%%%%%%%%%%%%%%%%%%%%%%%%%%%%%%%%%%%%%%%%%%%%%%%

This section formalizes the decision framework introduced in
Section~\ref{sec:problem}. We first define the network, model parameters, and
decision variables. We then formulate the geopolitical uncertainty set and
the operational recourse problem before presenting the complete robust
optimization model. Finally, we define restricted uncertainty environments
that isolate selected consequence channels.

%%%%%%%%%%%%%%%%%%%%%%%%%%%%%%%%%%%%%%%%%%%%%%%%%%%%%%%%%%%%%%%%%%%%%%%%%%%%%%%%
\subsection{Network structure, parameters, and decisions}
\label{sec:network}
%%%%%%%%%%%%%%%%%%%%%%%%%%%%%%%%%%%%%%%%%%%%%%%%%%%%%%%%%%%%%%%%%%%%%%%%%%%%%%%%

The supply chain is represented by a directed network
$\mathcal{G}=(N,E)$ with four echelons. The node set is $N=S\cup P\cup W\cup R$,
where $S$, $P$, $W$, and $R$ denote the sets of suppliers, manufacturing
plants, warehouses, and retailers, respectively. Material flows between
adjacent echelons, so $E=E^{SP}\cup E^{PW}\cup E^{WR}$, 
where $E^{SP}\subseteq S\times P$, $E^{PW}\subseteq P\times W$, and
$E^{WR}\subseteq W\times R$ contain the arcs connecting suppliers to plants,
plants to warehouses, and warehouses to retailers, respectively.
Let $C$ denote the set of countries represented in the network, and let
$\kappa(i)\in C$ identify the country in which node $i\in N$ is located.
Transportation arcs are classified according to the countries of their
endpoints:
\begin{equation}
\label{eq:international_domestic_link}
E^{I} = \{(i,j)\in E:\kappa(i)\neq\kappa(j)\},
\qquad E^{D} = \{(i,j)\in E:\kappa(i)=\kappa(j)\}
= E\setminus E^{I}.
\end{equation}
The sets $E^{I}$ and $E^{D}$ contain the international and domestic
transportation arcs, respectively. Tariff escalation and restrictions on
usable cross border capacity apply only to arcs in $E^{I}$, whereas physical
transportation disruptions may affect arcs in either set.
Supplier $s\in S$, manufacturing plant $p\in P$, and warehouse $w\in W$
have nominal capacities $U_s$, $U_p$, and $U_w$, respectively. Each
transportation arc $(i,j)\in E$ has nominal capacity $u_{ij}$ and unit
operating cost $c_{ij}$, excluding tariff payments. Each retailer $r\in R$
has demand $D_r$, and each unit of unmet demand at retailer $r$ incurs penalty
cost $\pi_r$.
For each international arc $(i,j)\in E^{I}$, let
$\underline{\tau}_{ij}$ denote the tariff rate under normal trade conditions
and let $\overline{\tau}_{ij}$ denote the elevated tariff rate under tariff
escalation. Both are ad valorem rates expressed as proportions of the customs
value of the goods shipped. Define the tariff increment as
\begin{equation}
\label{eq:tariff_equation}
\Delta\tau_{ij} = \overline{\tau}_{ij}
- \underline{\tau}_{ij},
\qquad (i,j)\in E^{I}.
\end{equation}
Let $V_{ij}$ denote the customs value per unit shipped on international arc
$(i,j)\in E^{I}$. The tariff payment per unit shipped on an international arc
is therefore the realized tariff rate multiplied by $V_{ij}$.
Restrictions on usable international transportation capacity are
parameterized by $\delta_{ij}\in[0,1]$ for $(i,j)\in E^{I}$, where
$\delta_{ij}$ denotes the fraction of nominal capacity lost when the
corresponding trade capacity restriction occurs. Thus,
$\delta_{ij}=0$ represents no capacity loss, $\delta_{ij}=1$ represents
complete loss of usable capacity through this channel, and values strictly
between zero and one represent partial capacity reductions.
Finally, $F_s$ denotes the cost of qualifying supplier $s\in S$, and
$K_{ij}$ denotes the cost of activating transportation arc $(i,j)\in E$.
The principal sets and parameters are summarized in
Table~\ref{tab:notation}.
\begin{table}[htbp]
\centering
\caption{Principal sets and model parameters.}
\label{tab:notation}
\small
\renewcommand{\arraystretch}{1.12}
\setlength{\tabcolsep}{6pt}

\begin{tabular}{p{0.22\linewidth} p{0.70\linewidth}}
\toprule
\textbf{Notation} & \textbf{Description} \\
\midrule

\multicolumn{2}{l}{\textit{\textbf{Sets and indices:}}} \\[1pt]

$S, P, W, R$
& Sets of suppliers, plants, warehouses, and retailers,
respectively \\

$C$
& Set of countries represented in the supply chain \\

$N$
& Set of supply chain nodes, $N=S\cup P\cup W\cup R$ \\

$E^{SP}, E^{PW}, E^{WR}$
& Sets of arcs connecting $S \rightarrow P$, $P \rightarrow W$, and
$W \rightarrow R$ respectively \\

$E$
& Set of all transportation arcs,
$E=E^{SP}\cup E^{PW}\cup E^{WR}$ \\

$E^{I}, E^{D}$
& Sets of international \& domestic transportation arcs, respectively \\

$\kappa(i)$
& Country in which node $i\in N$ is located \\

\hline
\addlinespace[3pt]
\multicolumn{2}{l}{\textit{\textbf{Network and demand parameters:}}} \\[1pt]

$U_s, U_p, U_w$
& Nominal capacities of $s\in S$, $p\in P$, and $w\in W$,
respectively \\

$u_{ij}$
& Nominal capacity of transportation arc $(i,j)\in E$ \\

$c_{ij}$
& Unit operating cost on arc $(i,j)\in E$,
excluding tariff payments \\

$D_r$
& Demand at retailer $r\in R$ \\

$\pi_r$
& Penalty cost per unit of unmet demand at retailer $r\in R$ \\

\hline
\addlinespace[3pt]
\multicolumn{2}{l}{\textit{\textbf{Trade policy parameters:}}} \\[1pt]

$\underline{\tau}_{ij}, \overline{\tau}_{ij}$
& Baseline and elevated tariff rates on international arc $(i,j)\in E^{I}$ \\

$\Delta\tau_{ij}$
& Tariff-rate increment,
$\Delta\tau_{ij}=\overline{\tau}_{ij}-\underline{\tau}_{ij}$ \\

$V_{ij}$
& Customs value per unit shipped on international arc $(i,j)\in E^{I}$ \\

$\delta_{ij}$
& Fraction of nominal capacity lost under a trade capacity restriction on
international arc $(i,j)\in E^{I}$ \\

\hline
\addlinespace[3pt]
\multicolumn{2}{l}{\textit{\textbf{Strategic cost parameters:}}} \\[1pt]

$F_s$
& Cost of qualifying supplier $s\in S$ \\

$K_{ij}$
& Cost of activating transportation arc $(i,j)\in E$ \\

\bottomrule
\end{tabular}

\vspace{2pt}
\begin{minipage}{0.92\linewidth}
\footnotesize
All capacities, demands, operating costs, qualification costs,
activation costs, and unmet demand penalties are assumed to be
nonnegative. Tariff rates and customs values are also nonnegative,
with $\overline{\tau}_{ij}\geq\underline{\tau}_{ij}$ for every
$(i,j)\in E^{I}$.
\end{minipage}

\end{table}

The strategic network configuration is selected before geopolitical
uncertainty is realized. For each supplier $s\in S$, let
$y_s\in\{0,1\}$ equal one if supplier $s$ is qualified and zero otherwise.
For each transportation arc $(i,j)\in E$, let $v_{ij}\in\{0,1\}$ equal one
if the arc is activated and zero otherwise. Define the feasible set of
strategic decisions as
\begin{equation}
\mathcal{Y}
=
\left\{
(y,v):
\begin{array}{l}
y_s\in\{0,1\}, \quad s\in S,\\
v_{ij}\in\{0,1\}, \quad (i,j)\in E,\\
v_{sp}\leq y_s, \quad (s,p)\in E^{SP}
\end{array}
\right\},
\label{eq:first_stage_domain}
\end{equation}
where $v_{sp}\leq y_s$ ensures that an arc connecting supplier $s$ to plant
$p$ can be activated only when supplier $s$ is qualified. Thus,
$(y,v)\in\mathcal{Y}$ determines the suppliers and transportation arcs
available for subsequent operations. The corresponding strategic
reconfiguration cost is
\begin{equation}
\label{eq:config_cost}
C^{R}(y,v)
=
\sum_{s\in S}F_s y_s
+
\sum_{(i,j)\in E}K_{ij}v_{ij}.
\end{equation}
After geopolitical uncertainty is realized, operational flows adapt to the
resulting facility availability, transportation availability, and trade
conditions. For each arc $(i,j)\in E$, let $x_{ij}\geq 0$ denote the
quantity shipped on that arc after the uncertainty realization is observed.
For each retailer $r\in R$, let $\ell_r\geq 0$ denote unmet demand. The
variables $(y,v)$ are therefore strategic decisions made before uncertainty
is observed, whereas $(x,\ell)$ are adaptive operational decisions determined
after the geopolitical realization. The uncertainty variables governing that
realization are defined next.

%%%%%%%%%%%%%%%%%%%%%%%%%%%%%%%%%%%%%%%%%%%%%%%%%%%%%%%%%%%%%%%%%%%%%%%%%%%%%%%%
\subsection{Multi-channel geopolitical uncertainty set}
\label{sec:uncertainty}
%%%%%%%%%%%%%%%%%%%%%%%%%%%%%%%%%%%%%%%%%%%%%%%%%%%%%%%%%%%%%%%%%%%%%%%%%%%%%%%%

The multi-channel geopolitical uncertainty described in
Section~\ref{sec:problem} is formalized through a country-aware uncertainty
set. For each country $c\in C$, define the binary country-shock variable $g_c \in \{0,1\}$, such that  $g_c = 1$, if country $c$ is affected by a geopolitical shock and zero otherwise.
A geopolitical shock represents an adverse country-level event capable of
altering trade conditions or disrupting supply chain operations associated
with that country. The set of simultaneously affected countries is restricted
by
\begin{equation}
\sum_{c\in C}g_c\leq\Gamma_G,
\label{eq:country_budget}
\end{equation}
where $\Gamma_G$ denotes the maximum number of countries that may be affected
in an admissible uncertainty realization. The consequences of the country-level shocks are represented by six families of binary uncertainty variables. For suppliers, manufacturing plants, and warehouses, respectively, we define the binary facility-disruption variables as:
\begin{equation}
a_s=
\begin{cases}
1, & \text{if $s$ disrupts},\\
0, & \text{otherwise},
\end{cases}
\quad
b_p=
\begin{cases}
1, & \text{if $p$ disrupts},\\
0, & \text{otherwise},
\end{cases}
\quad
d_w=
\begin{cases}
1, & \text{if $w$ disrupts},\\
0, & \text{otherwise},
\end{cases}
\label{eq:facility_disruption_variables}
\end{equation}
for $s\in S$, $p\in P$, and $w\in W$. For transportation and trade-policy consequences, let $z_{ij}\in\{0,1\}$, $(i,j)\in E$, denote the physical-disruption indicator, with $z_{ij}=1$ if transportation arc $(i,j)$ is physically unavailable and $z_{ij}=0$ otherwise. For each international arc $(i,j)\in E^{I}$, let $q_{ij},h_{ij}\in\{0,1\}$. The variable $q_{ij}$ identifies the applicable tariff regime, with $q_{ij}=0$ corresponding to the baseline tariff rate $\underline{\tau}_{ij}$ and $q_{ij}=1$ to the elevated tariff rate $\overline{\tau}_{ij}$. The variable $h_{ij}$ indicates whether a
policy-induced cross-border capacity restriction is active, with $h_{ij}=1$
when the restriction applies and $h_{ij}=0$ otherwise.
The variables $z_{ij}$ and $h_{ij}$ represent distinct disruption mechanisms.
While $z_{ij}=1$ renders transportation arc $(i,j)$ physically unavailable,
$h_{ij}=1$ activates the fractional cross-border capacity reduction
$\delta_{ij}$ defined in Section~\ref{sec:network}. Thus, $h_{ij}$ may reduce
the usable capacity of an international arc without physically disabling the
underlying transportation connection.

The network-level consequences are linked to the country-shock variables
through the location mapping $\kappa(\cdot)$ defined in
Section~\ref{sec:network}. Facility disruptions are restricted by
\begin{equation}
a_s \leq g_{\kappa(s)} \ \forall  s\in S; \quad 
b_p \leq g_{\kappa(p)} \ \forall  p\in P; \quad 
d_w \leq g_{\kappa(w)}  \ \forall  w\in W,
\label{eq:country_facility}
\end{equation}
while physical transportation failures and trade-policy consequences satisfy
\begin{align}
z_{ij}
&\leq g_{\kappa(i)}+g_{\kappa(j)},
&& (i,j)\in E,
\label{eq:country_transport}\\
q_{ij}
&\leq g_{\kappa(i)}+g_{\kappa(j)},
&& (i,j)\in E^{I},
\label{eq:country_tariff}\\
h_{ij}
&\leq g_{\kappa(i)}+g_{\kappa(j)},
&& (i,j)\in E^{I}.
\label{eq:country_tradecapacity}
\end{align}
Constraint~\eqref{eq:country_facility} models the geographical dependence of
facility disruptions by permitting a supplier, plant, or warehouse to be
disrupted only when its host country experiences a geopolitical shock.
Constraint~\eqref{eq:country_transport} extends this dependence to the
transportation network: physical failure of arc $(i,j)$ is admissible only
when at least one of its endpoint countries is affected.
Constraints~\eqref{eq:country_tariff} and
\eqref{eq:country_tradecapacity} impose the corresponding geographical
conditions on the two trade-policy channels, so that tariff escalation or a
policy-induced capacity restriction on an international arc can occur only
when its origin or destination country is affected. These implications are
intentionally one-directional. In particular, $g_c=1$ identifies the
facilities and arcs exposed to a shock in country $c$ but does not require
every exposed component to experience an adverse consequence. The formulation
therefore permits heterogeneous network-level consequences conditional on the
same country-level shock realization. Different consequence channels may also affect the same international transportation arc simultaneously, provided that the geographic coupling constraints and uncertainty budgets are satisfied.
The extent of the admissible consequences is further controlled through
channel-specific uncertainty budgets. Facility and physical transportation
disruptions satisfy
\begin{align}
\sum_{s\in S}a_s &\leq \Gamma_S,
&
\sum_{p\in P}b_p &\leq \Gamma_P,
&
\sum_{w\in W}d_w &\leq \Gamma_W,
\label{eq:facility_budgets}\\
\sum_{(s,p)\in E^{SP}}z_{sp} &\leq \Gamma_{SP},
&
\sum_{(p,w)\in E^{PW}}z_{pw} &\leq \Gamma_{PW},
&
\sum_{(w,r)\in E^{WR}}z_{wr} &\leq \Gamma_{WR},
\label{eq:transport_budgets}
\end{align}
and the trade-policy channels satisfy
\begin{equation}
\sum_{(i,j)\in E^{I}}q_{ij}\leq\Gamma_T,
\qquad
\sum_{(i,j)\in E^{I}}h_{ij}\leq\Gamma_H.
\label{eq:trade_budgets}
\end{equation}
Constraint~\eqref{eq:facility_budgets} bounds the numbers of simultaneously
disrupted suppliers, plants, and warehouses by $\Gamma_S$, $\Gamma_P$, and
$\Gamma_W$, respectively. Constraint~\eqref{eq:transport_budgets} limits
physical transportation failures separately within the supplier--plant,
plant--warehouse, and warehouse--retailer echelons through
$\Gamma_{SP}$, $\Gamma_{PW}$, and $\Gamma_{WR}$, respectively. This
echelon-specific structure allows the uncertainty set to represent different
levels of physical disruption exposure across the three transportation
layers. Constraint~\eqref{eq:trade_budgets} bounds the numbers of
international arcs simultaneously subject to tariff escalation and
policy-induced capacity restrictions by $\Gamma_T$ and $\Gamma_H$
respectively.

The preceding constraints regulate the extent of each consequence type
separately. To additionally limit their simultaneous occurrence across
channels, we impose the aggregate uncertainty budget:
\begin{equation}
\omega_F \Bigl(\sum_{s\in S}a_s + \sum_{p\in P}b_p + \sum_{w\in W}d_w
+ \sum_{(i,j)\in E}z_{ij} \Bigr)
+ \omega_T\sum_{(i,j)\in E^{I}}q_{ij}
+ \omega_H\sum_{(i,j)\in E^{I}}h_{ij}
\leq \Gamma_J.
\label{eq:joint_budget}
\end{equation}
The parameters $\omega_F$, $\omega_T$, and $\omega_H$ specify the
amount of the aggregate uncertainty budget consumed by a physical
disruption, tariff escalation, and trade capacity restriction,
respectively. Constraint~\eqref{eq:joint_budget} therefore limits the
combinations of adverse consequences that may occur simultaneously
across the different channels. %The individual budgets in \eqref{eq:facility_budgets} through \eqref{eq:trade_budgets} restrict the number of consequences within each category, while the aggregate budget restricts their combined occurrence.
%
%
%
%
%Together with the country-level budget in Constraint~\eqref{eq:country_budget}, these restrictions define the admissible geopolitical scenarios. 
Let $\xi=(g,a,b,d,z,q,h)$ denote a complete uncertainty realization. The full country-aware multi-channel uncertainty set is obtained as:
\begin{equation}
\mathcal{U}^{F}
=
\left\{
\xi:
\begin{array}{l}
g_c,a_s,b_p,d_w,z_{ij},q_{ij},h_{ij}\in\{0,1\},\\[1mm]
\eqref{eq:country_budget},
\eqref{eq:country_facility}-\eqref{eq:joint_budget}\ \text{are satisfied}
\end{array}
\right\},
\label{eq:uncertainty_set}
\end{equation}
where $c\in C$, $s\in S$, $p\in P$, $w\in W$, and $(i,j)\in E$ for $z_{ij}$, while $(i,j)\in E^{I}$ for $q_{ij}$ and $h_{ij}$.
%Accordingly, $g$, $(a,b,d)$, $z$, $q$, and $h$ represent the country-shock, facility-disruption, physical transportation-disruption, tariff-escalation, and policy-induced cross-border capacity-restriction variables, respectively.
%
The uncertainty set~\eqref{eq:uncertainty_set} distinguishes the
geographic origin of geopolitical uncertainty, represented by $g$,
from its consequences for the supply chain, represented by
$(a,b,d,z,q,h)$. %Constraint~\eqref{eq:country_budget} limits the number of countries affected by geopolitical shocks, while Constraints~\eqref{eq:country_facility} through \eqref{eq:country_tradecapacity} determine which network consequences are geographically admissible. Constraints~\eqref{eq:facility_budgets} through \eqref{eq:trade_budgets} limit the number of consequences within individual categories, and Constraint~\eqref{eq:joint_budget} restricts their simultaneous occurrence across categories. 
For a fixed realization $\xi\in\mathcal{U}^{F}$, the corresponding operational recourse problem is formulated in Section~\ref{sec:recourse}.

%%%%%%%%%%%%%%%%%%%%%%%%%%%%%%%%%%%%%%%%%%%%%%%%%%%%%%%%%%%%%%%%%%%%%%%%%%%%%%%%
\subsection{Operational recourse model}
\label{sec:recourse}
%%%%%%%%%%%%%%%%%%%%%%%%%%%%%%%%%%%%%%%%%%%%%%%%%%%%%%%%%%%%%%%%%%%%%%%%%%%%%%%%

For a given uncertainty realization $\xi\in\mathcal{U}^{F}$, we define the
total second-stage operating cost associated with recourse decisions $(x,\ell)$ as:
\begin{equation}
C^{O}(x,\ell;\xi)
=
\sum_{(i,j)\in E}c_{ij}x_{ij}
+
\sum_{(i,j)\in E^{I}}
\left(
\underline{\tau}_{ij}
+
\Delta\tau_{ij}q_{ij}
\right)V_{ij}x_{ij}
+
\sum_{r\in R}\pi_r\ell_r.
\label{eq:operating_cost}
\end{equation}
For fixed first-stage decisions $(y,v)$ and an uncertainty realization
$\xi\in\mathcal{U}^{F}$, the firm determines the recourse decisions
$(x,\ell)$ that minimize~\eqref{eq:operating_cost} subject to the realized
facility, transportation, and trade-policy conditions. The resulting
operational recourse problem is
\begin{align}
Q(y,v,\xi)
=
\min_{x,\ell}\quad
& C^{O}(x,\ell;\xi)
\label{eq:recourse_obj}\\
\text{s.t.}\quad
&
\sum_{p:(s,p)\in E^{SP}}x_{sp}
\leq
U_s y_s(1-a_s),
&& s\in S,
\label{eq:supplier_capacity}\\
&
\sum_{s:(s,p)\in E^{SP}}x_{sp}
=
\sum_{w:(p,w)\in E^{PW}}x_{pw},
&& p\in P,
\label{eq:plant_balance}\\
&
\sum_{w:(p,w)\in E^{PW}}x_{pw}
\leq
U_p(1-b_p),
&& p\in P,
\label{eq:plant_capacity}\\
&
\sum_{p:(p,w)\in E^{PW}}x_{pw}
=
\sum_{r:(w,r)\in E^{WR}}x_{wr},
&& w\in W,
\label{eq:warehouse_balance}\\
&
\sum_{r:(w,r)\in E^{WR}}x_{wr}
\leq
U_w(1-d_w),
&& w\in W,
\label{eq:warehouse_capacity}\\
&
\sum_{w:(w,r)\in E^{WR}}x_{wr}
+
\ell_r
=
D_r,
&& r\in R,
\label{eq:demand_balance}\\
&
x_{ij}
\leq
u_{ij}v_{ij}(1-z_{ij}),
&& (i,j)\in E,
\label{eq:physical_capacity}\\
&
x_{ij}
\leq
u_{ij}v_{ij}
\left(1-\delta_{ij}h_{ij}\right),
&& (i,j)\in E^{I},
\label{eq:trade_capacity}\\
&
x_{ij}\geq0,
&& (i,j)\in E,
\label{eq:flow_nonnegative}\\
&
\ell_r\geq0,
&& r\in R.
\label{eq:shortage_bounds}
\end{align}
%
%
%
%
%
%
\begin{comment}
\begin{align}
Q(y,v,\xi) = \min_{x,\ell}\quad & \sum_{(i,j)\in E}c_{ij}x_{ij} + \sum_{(i,j)\in E^{I}}
\left(\underline{\tau}_{ij} + \Delta\tau_{ij}q_{ij} \right)V_{ij}x_{ij} + && \sum_{r\in R}\pi_r\ell_r \label{eq:recourse_obj}
\\
&
\sum_{p:(s,p)\in E^{SP}}x_{sp} \leq U_s y_s(1-a_s), && s\in S, \label{eq:supplier_capacity}
\\
&
\sum_{s:(s,p)\in E^{SP}}x_{sp}
=
\sum_{w:(p,w)\in E^{PW}}x_{pw},
&& p\in P,
\label{eq:plant_balance}\\
&
\sum_{w:(p,w)\in E^{PW}}x_{pw}
\leq
U_p(1-b_p),
&& p\in P,
\label{eq:plant_capacity}\\
&
\sum_{p:(p,w)\in E^{PW}}x_{pw}
=
\sum_{r:(w,r)\in E^{WR}}x_{wr},
&& w\in W,
\label{eq:warehouse_balance}\\
&
\sum_{r:(w,r)\in E^{WR}}x_{wr}
\leq
U_w(1-d_w),
&& w\in W,
\label{eq:warehouse_capacity}\\
&
\sum_{w:(w,r)\in E^{WR}}x_{wr}
+
\ell_r
=
D_r,
&& r\in R,
\label{eq:demand_balance}\\
&
x_{ij}
\leq
u_{ij}v_{ij}(1-z_{ij}),
&& (i,j)\in E,
\label{eq:physical_capacity}\\
&
x_{ij}
\leq
u_{ij}v_{ij}
\left(1-\delta_{ij}h_{ij}\right),
&& (i,j)\in E^{I},
\label{eq:trade_capacity}\\
&
x_{ij}\geq0,
&& (i,j)\in E,
\label{eq:flow_nonnegative}\\
&
\ell_r \geq 0,
&& r\in R.
\label{eq:shortage_bounds}
\end{align}
\end{comment}
%
%
%
%
%
%
%
%
%
%
%
The objective function in~\eqref{eq:recourse_obj} consists of three cost
components. The first term,
$\sum_{(i,j)\in E}c_{ij}x_{ij}$, is the operating cost associated with
shipments over the transportation network, where $c_{ij}$ is the unit
operating cost on arc $(i,j)$. The second term, $\sum_{(i,j)\in E^{I}}
\left(\underline{\tau}_{ij}+\Delta\tau_{ij}q_{ij}\right)
V_{ij}x_{ij}$, is the tariff cost incurred on international shipments. For each
$(i,j)\in E^{I}$, the realized tariff rate is
$\underline{\tau}_{ij}+\Delta\tau_{ij}q_{ij}$, which, by
Equation~\eqref{eq:tariff_equation}, equals the baseline rate
$\underline{\tau}_{ij}$ when $q_{ij}=0$ and the elevated rate
$\overline{\tau}_{ij}$ when $q_{ij}=1$. Thus, $\left(\underline{\tau}_{ij}+\Delta\tau_{ij}q_{ij}\right)V_{ij}$ represents the tariff payment per unit shipped on international arc $(i,j)$, and multiplication by $x_{ij}$ yields the total tariff payment associated
with shipments on that arc. The third term, $\sum_{r\in R}\pi_r\ell_r$, is the shortage-penalty cost, where $\pi_r$ is the unit penalty applied to unmet demand $\ell_r$ at retailer $r$. Accordingly, $Q(y,v,\xi)$ minimizes the total second-stage cost of operating
the network under realization $\xi$.

Constraint~\eqref{eq:supplier_capacity} limits the total outbound shipment
from supplier $s$ to its realized available capacity. The product
$y_s(1-a_s)$ captures the joint effect of the first-stage qualification
decision and the realized disruption state: positive supply is possible only
when supplier $s$ has been qualified ($y_s=1$) and remains operational
($a_s=0$). 
Constraint~\eqref{eq:plant_balance} enforces flow conservation at each
manufacturing plant by equating total inbound shipments from suppliers with
total outbound shipments to warehouses, while
Constraint~\eqref{eq:plant_capacity} limits the latter by the plant's
realized processing capacity $U_p(1-b_p)$. Similarly,
Constraint~\eqref{eq:warehouse_balance} equates total inbound shipments to
each warehouse with its outbound shipments to retailers, and
Constraint~\eqref{eq:warehouse_capacity} limits total outbound flow by the
realized warehouse capacity $U_w(1-d_w)$. %These balance equations assume unit conversion of flow across successive echelons.
Finally, Constraint~\eqref{eq:demand_balance} requires the demand $D_r$ at
each retailer $r$ to be accounted for by incoming shipments and unmet demand
$\ell_r$. Thus, any demand that cannot be served through the realized network
is explicitly represented by $\ell_r$ and incurs the penalty specified in
the objective function~\eqref{eq:recourse_obj}.

Constraint~\eqref{eq:physical_capacity} restricts flow on arc $(i,j)$ to its
realized physical capacity, conditional on the first-stage activation
decision. Specifically, $v_{ij}=0$ prevents flow on an arc that was not
activated, while $z_{ij}=1$ reduces the available capacity of an activated
arc to zero. For each international arc $(i,j)\in E^{I}$,
Constraint~\eqref{eq:trade_capacity} imposes an additional capacity bound
associated with policy-induced trade restrictions. When $h_{ij}=0$, this
bound equals the nominal activated capacity $u_{ij}v_{ij}$; when $h_{ij}=1$,
the usable cross-border capacity is reduced to
$(1-\delta_{ij})u_{ij}v_{ij}$. Consequently, an international arc must
simultaneously satisfy the physical-capacity and trade-capacity constraints.
A physical disruption ($z_{ij}=1$) renders the arc unavailable regardless of
its trade-restriction status, whereas a policy-induced restriction
($h_{ij}=1$) reduces capacity by the fraction $\delta_{ij}$ when the arc
remains physically available. This formulation preserves the distinction
between the physical-transportation and trade-capacity channels defined in
Section~\ref{sec:uncertainty}.
Finally, Constraints~\eqref{eq:flow_nonnegative} and
\eqref{eq:shortage_bounds} specify the domains of the recourse variables.
Because unmet demand is permitted up to $D_r$ at every retailer,
the recourse problem is feasible for every first-stage configuration and every
$\xi\in\mathcal{U}^{F}$. %Hence, the model has relatively complete recourse, a property used in the reformulation developed in Section~\ref{sec:reformulation}.

%%%%%%%%%%%%%%%%%%%%%%%%%%%%%%%%%%%%%%%%%%%%%%%%%%%%%%%%%%%%%%%%%%%%%%%%%%%%%%%%
\subsection{Two-stage robust formulation \& nested uncertainty models}
\label{sec:robust}
%%%%%%%%%%%%%%%%%%%%%%%%%%%%%%%%%%%%%%%%%%%%%%%%%%%%%%%%%%%%%%%%%%%%%%%%%%%%%%%%

Combining the first-stage reconfiguration cost in~\eqref{eq:config_cost}, the multi-channel uncertainty set
$\mathcal{U}^{F}$ in~\eqref{eq:uncertainty_set}, and the operational recourse
function $Q(y,v,\xi)$ in~\eqref{eq:recourse_obj}--\eqref{eq:shortage_bounds}
yields the two-stage robust optimization problem:
\begin{equation}
Z^{F} = \min_{(y,v)\in\mathcal{Y}}\left\{C^{R}(y,v) + \max_{\xi\in\mathcal{U}^{F}}
Q(y,v,\xi) \right\}, \label{eq:robust_model}
\end{equation}
where $\mathcal{Y}$ denotes the feasible domain of the first-stage strategic
decisions. Equivalently, using the definition of the recourse function,
Equation~\eqref{eq:robust_model} has the following min--max--min representation:
\begin{equation}
Z^{F} = \min_{(y,v)\in\mathcal{Y}} \left\{C^{R}(y,v) + \max_{\xi\in\mathcal{U}^{F}} \;\min_{(x,\ell)\in\mathcal{F}(y,v,\xi)} C^{O}(x,\ell;\xi) \right\},\label{eq:minmaxmin}
\end{equation}
where $\mathcal{F}(y,v,\xi)$ denotes the feasible region defined by
Constraints~\eqref{eq:supplier_capacity}--\eqref{eq:shortage_bounds}, and $C^{O}(x,\ell;\xi)$ denotes the second-stage objective function defined
in~\eqref{eq:operating_cost} and used in~\eqref{eq:recourse_obj}. The outer minimization selects the strategic
network configuration before uncertainty is observed; the maximization
identifies the most adverse realization in $\mathcal{U}^{F}$ for that
configuration; and the inner minimization determines the optimal operational
response after the realization is observed.
An important structural property of the proposed formulation is that
restricted uncertainty environments can be obtained directly from the full
multi-channel uncertainty set $\mathcal{U}^{F}$ by suppressing selected
consequence channels, while preserving the underlying supply chain network,
first-stage decision structure, and operational recourse model. The resulting
formulations therefore differ only in the channels through which
country-level geopolitical shocks are permitted to affect the network. This
provides a common modeling basis for isolating the strategic implications of
individual uncertainty channels and, subsequently, for assessing the
consequences of omitting selected channels from the robust design problem.
We consider three such restrictions: tariff-only, trade-policy, and
physical-disruption uncertainty.
Let $\mathcal{U}^{T}$ denote the tariff-only uncertainty set obtained from
$\mathcal{U}^{F}$ by imposing:
\begin{equation}
\mathcal{U}^{T} = \left\{\xi\in\mathcal{U}^{F}:a=b=d=z=h=0\right\}.
\label{eq:tariff_only_set}
\end{equation}
Under $\mathcal{U}^{T}$, country-level shocks may propagate to the network
only through tariff escalation $q$; facility disruptions, physical
transportation failures, and policy-induced cross-border capacity
restrictions are excluded. 
The trade-policy uncertainty set is defined as:
\begin{equation}
\mathcal{U}^{TP} = \left\{\xi\in\mathcal{U}^{F}:a=b=d=z=0\right\},
\label{eq:trade_policy_set}
\end{equation}
which retains both trade-policy channels, $q$ and $h$. Thus, geopolitical
shocks may alter the tariff cost and usable cross-border capacity of
international arcs, while facility availability and physical transportation
connectivity remain unaffected.
Finally, the physical-disruption uncertainty set is:
\begin{equation}
\mathcal{U}^{D} = \left\{\xi\in\mathcal{U}^{F}:q=h=0 \right\},\label{eq:disruption_only_set}
\end{equation}
under which geopolitical shocks may propagate through facility disruptions
$(a,b,d)$ and physical transportation failures $z$, while the trade-policy
channels are suppressed. In each restricted uncertainty set, the applicable
country-coupling and uncertainty-budget constraints are inherited from
$\mathcal{U}^{F}$. These restrictions induce the following nesting
relationship among the uncertainty sets.
\begin{proposition}
\label{prop:nesting}
The tariff-only, trade-policy, and physical-disruption uncertainty sets defined
in~\eqref{eq:tariff_only_set}--\eqref{eq:disruption_only_set} satisfy
\begin{equation}
\mathcal{U}^{T} \subseteq \mathcal{U}^{TP} \subseteq \mathcal{U}^{F}, \qquad \mathcal{U}^{D}
\subseteq \mathcal{U}^{F}.
\label{eq:set_nesting}
\end{equation}
\end{proposition}
\begin{proof}
The set $\mathcal{U}^{TP}$ is obtained from $\mathcal{U}^{F}$ by imposing
$a=b=d=z=0$, and $\mathcal{U}^{T}$ is obtained from
$\mathcal{U}^{TP}$ by additionally imposing $h=0$. Hence,
$\mathcal{U}^{T}\subseteq\mathcal{U}^{TP}\subseteq\mathcal{U}^{F}$.
Similarly, imposing $q=h=0$ on $\mathcal{U}^{F}$ yields
$\mathcal{U}^{D}$, and therefore
$\mathcal{U}^{D}\subseteq\mathcal{U}^{F}$.
\end{proof}

Proposition~\ref{prop:nesting} establishes that the restricted uncertainty
environments are embedded within the full multi-channel uncertainty
architecture while preserving the underlying network and decision structure.
This nesting permits strategic configurations obtained under different
uncertainty environments to be compared within a common modeling framework
and subsequently evaluated under the full multi-channel uncertainty set.
For any fixed strategic configuration $(y,v)$,
Proposition~\ref{prop:nesting} implies $\max_{\xi\in\mathcal{U}^{T}}Q(y,v,\xi)
\leq \max_{\xi\in\mathcal{U}^{TP}}Q(y,v,\xi)
\leq \max_{\xi\in\mathcal{U}^{F}}Q(y,v,\xi)$, and $\max_{\xi\in\mathcal{U}^{D}}Q(y,v,\xi) \leq \max_{\xi\in\mathcal{U}^{F}}Q(y,v,\xi)$.
Because the feasible set of strategic decisions and the reconfiguration
cost remain unchanged across uncertainty environments, these
relationships also imply $Z^{T}\leq Z^{TP}\leq Z^{F}$, and $Z^{D}\leq Z^{F}$,
where $Z^{T}$, $Z^{TP}$, and $Z^{D}$ denote the optimal robust
objective values under their respective uncertainty sets.
These inequalities follow directly from the nesting relationships.
However, the ordering of the optimal robust objective values does not
imply that the corresponding strategic configurations are nested or
that their reconfiguration costs increase as the uncertainty set
expands.
%
%
%
%
%
%
%For any fixed strategic configuration $(y,v)$, Proposition~\ref{prop:nesting} further implies: $\max_{\xi\in\mathcal{U}^{T}}Q(y,v,\xi) \leq \max_{\xi\in\mathcal{U}^{TP}}Q(y,v,\xi) \leq \max_{\xi\in\mathcal{U}^{F}}Q(y,v,\xi)$, and $\max_{\xi\in\mathcal{U}^{D}}Q(y,v,\xi) \leq \max_{\xi\in\mathcal{U}^{F}}Q(y,v,\xi)$. These inequalities follow directly from the set inclusions in~\eqref{eq:set_nesting}: for fixed $(y,v)$, maximizing the same recourse function over a larger uncertainty set cannot decrease its worst-case value. This ordering applies to the worst-case recourse cost of a fixed strategic configuration and does not imply an ordering of the optimal first-stage configurations or their associated reconfiguration costs, since $(y,v)$ is reoptimized under each uncertainty environment.
%
%
%
%
%
Equation~\eqref{eq:minmaxmin} captures the sequential interaction between the firm and nature. %The outer minimization determines the strategic network configuration before uncertainty is realized. Nature subsequently selects the most adverse country-aware geopolitical scenario contained within the uncertainty set. Finally, after observing the realized uncertainty, the firm adaptively determines sourcing, production, transportation, and distribution decisions that minimize operational cost over the surviving network.
The resulting formulation constitutes a mixed-integer two-stage robust optimization problem with adaptive recourse. Owing to the embedded max--min structure, directly solving \eqref{eq:minmaxmin} is computationally challenging. The following section uses strong duality to eliminate the inner recourse minimization for fixed first-stage decisions and develops the adversarial subproblem used within a column-and-constraint generation procedure. %The resulting adversarial formulation is a nonconvex mixed-integer bilinear optimization problem.

%%%%%%%%%%%%%%%%%%%%%%%%%%%%%%%%%%%%%%%%%%%%%%%%%%%%%%%%%%%%%%%%%%%%%%%%%%%%%%%%
\section{Reformulation and solution method}
\label{sec:reformulation}
%%%%%%%%%%%%%%%%%%%%%%%%%%%%%%%%%%%%%%%%%%%%%%%%%%%%%%%%%%%%%%%%%%%%%%%%%%%%%%%%

The robust formulation~\eqref{eq:robust_model} requires evaluation of the
worst-case recourse cost over $\mathcal{U}^{F}$ for a given first-stage
configuration. Explicit incorporation of all admissible uncertainty
realizations becomes computationally impractical as the dimensions of the
network and uncertainty set increase. We therefore employ the column and constraint generation (CCG) method (see for instance \cite{ZengZhao2013}), which iteratively augments a restricted master problem (RMP) with uncertainty
realizations identified by an adversarial subproblem. For fixed first-stage decisions, the adversarial subproblem is obtained by dualizing the linear operational recourse problem defined in Section~\ref{sec:recourse}.

%%%%%%%%%%%%%%%%%%%%%%%%%%%%%%%%%%%%%%%%%%%%%%%%%%%%%%%%%%%%%%%%%%%%%%%%%%%%%%%%
\subsection{Recourse dualization and adversarial problem}
\label{sec:dualization}
%%%%%%%%%%%%%%%%%%%%%%%%%%%%%%%%%%%%%%%%%%%%%%%%%%%%%%%%%%%%%%%%%%%%%%%%%%%%%%%%

For a fixed first-stage configuration $(\bar y,\bar v)\in\mathcal{Y}$, define
the corresponding worst-case recourse value as
\begin{equation}
\mathcal{Q}(\bar y,\bar v)
=
\max_{\xi\in\mathcal{U}^{F}}
Q(\bar y,\bar v,\xi).
\label{eq:worst_case_recourse}
\end{equation}
For any fixed $(\bar y,\bar v)\in\mathcal{Y}$ and
$\xi\in\mathcal{U}^{F}$, the first-stage decisions and all components of the
uncertainty realization are fixed parameters in the recourse problem defined
in Section~\ref{sec:recourse}. Consequently,
$Q(\bar y,\bar v,\xi)$ is a linear program in the shipment variables
$x_{ij}$, $(i,j)\in E$, and unmet-demand variables $\ell_r$, $r\in R$.
For every $(\bar y,\bar v)\in\mathcal{Y}$ and
$\xi\in\mathcal{U}^{F}$, the recourse problem is feasible. To establish this, consider the recourse solution: $x_{ij}=0,\ \forall (i,j)\in E$ and $\ell_r=D_r,\ \forall r\in R$,
%
%
%
%With $x_{ij}=0$ for all $(i,j)\in E$, the flow-balance Constraints~\eqref{eq:plant_balance} and~\eqref{eq:warehouse_balance} are satisfied. The supplier-, plant-, and warehouse-capacity Constraints~\eqref{eq:supplier_capacity}, \eqref{eq:plant_capacity}, and \eqref{eq:warehouse_capacity} are also satisfied because their left-hand sides are zero and their right-hand sides are nonnegative for every $(\bar y,\bar v)\in\mathcal{Y}$ and $\xi\in\mathcal{U}^{F}$.
%
%
%For every transportation arc $(i,j)\in E$, Constraint~\eqref{eq:physical_capacity} is satisfied by
%\eqref{eq:recourse_feasible_solution} because $0
%\leq u_{ij}\bar v_{ij}(1-z_{ij})$, where $u_{ij}\geq0$, $\bar v_{ij}\in\{0,1\}$, and $z_{ij}\in\{0,1\}$. For every international arc $(i,j)\in E^{I}$, Constraint~\eqref{eq:trade_capacity} is simultaneously satisfied because $0
%\leq u_{ij}\bar v_{ij} \left(1-\delta_{ij}h_{ij}\right)$, where $\delta_{ij}\in[0,1]$ and $h_{ij}\in\{0,1\}$ imply $1-\delta_{ij}h_{ij}\geq0$. Finally, substituting $x_{wr}=0$ and $\ell_r=D_r$ into
%Constraint~\eqref{eq:demand_balance} satisfies every retailer demand-balance constraint. 
for which $Q(\bar y,\bar v,\xi)$ is feasible for every
$(\bar y,\bar v)\in\mathcal{Y}$ and every $\xi\in\mathcal{U}^{F}$.
The recourse feasible region is also bounded. For every $(i,j)\in E$,
Constraint~\eqref{eq:physical_capacity} gives $0 \leq x_{ij} \leq u_{ij}\bar v_{ij}(1-z_{ij})$, and the right-hand side is finite. For an international arc $(i,j)\in E^{I}$, Constraint~\eqref{eq:trade_capacity} additionally imposes $x_{ij} \leq u_{ij}\bar v_{ij} \left(1-\delta_{ij}h_{ij}\right)$.
Furthermore, Constraint~\eqref{eq:demand_balance}, together with
$x_{wr}\geq0$ and $\ell_r\geq0$, implies $0\leq\ell_r\leq D_r$, for every $r\in R$.
Thus, all recourse variables are bounded over the feasible region. Since the coefficients of the recourse objective~\eqref{eq:recourse_obj} are finite, $Q(\bar y,\bar v,\xi)$ has a finite optimal value for every
$(\bar y,\bar v)\in\mathcal{Y}$ and $\xi\in\mathcal{U}^{F}$. Strong duality therefore applies to the recourse problem for every such pair.

To construct the recourse dual, we associate nonnegative dual variables
$\{\phi_s\}_{s\in S}$, $\{\gamma_p\}_{p\in P}$, and $\{\theta_w\}_{w\in W}$, with Constraints~\eqref{eq:supplier_capacity}, \eqref{eq:plant_capacity}, and \eqref{eq:warehouse_capacity}, respectively. 
Let $\{\alpha_{ij}\}_{(i,j)\in E} \geq0$, be the dual variables associated with the Constraints in~\eqref{eq:physical_capacity}, and let $\{\beta_{ij}\}_{(i,j)\in E^{I}}\geq0$, be the dual variable associated with the Constraints in~\eqref{eq:trade_capacity}. The equality
Constraints in~\eqref{eq:plant_balance}, \eqref{eq:warehouse_balance}, and
\eqref{eq:demand_balance} are assigned dual variables
$\{\rho_p\}_{p\in P}$, $\{\psi_w\}_{w\in W}$, and $\{\lambda_r\}_{r\in R}$,
respectively, each unrestricted in sign.

For each fixed $(\bar y,\bar v)\in\mathcal{Y}$ and
$\xi\in\mathcal{U}^{F}$, strong duality permits $Q(\bar y,\bar v,\xi)$ in~\eqref{eq:worst_case_recourse} to be replaced by
the optimal value of its dual. The dual objective is obtained from the right-hand sides of the capacity and demand-balance constraints \eqref{eq:supplier_capacity},
\eqref{eq:plant_capacity}, \eqref{eq:warehouse_capacity},
\eqref{eq:demand_balance}, \eqref{eq:physical_capacity}, and
\eqref{eq:trade_capacity}. Maximizing this dual representation jointly over
the admissible uncertainty realization and the recourse-dual variables gives:
\begin{align}
\mathcal{Q}(\bar y,\bar v) = \max_{\substack{
\xi\in\mathcal{U}^{F},
\\
\phi,\gamma,\theta,\alpha,\beta,\rho,\psi,\lambda}}
\quad
&
\sum_{r\in R}D_r\lambda_r - \sum_{s\in S}U_s\bar y_s(1-a_s)\phi_s
- \sum_{p\in P}U_p(1-b_p)\gamma_p
\nonumber
\\
& - \sum_{w\in W}U_w(1-d_w)\theta_w - \sum_{(i,j)\in E} u_{ij}\bar v_{ij}(1-z_{ij})\alpha_{ij}
\nonumber
\\
& - \sum_{(i,j)\in E^{I}} u_{ij}\bar v_{ij} \left(1-\delta_{ij}h_{ij}\right)\beta_{ij}.
\label{eq:adversarial_objective}
\end{align}
%
%where the first term is induced by the retailer demand-balance constraints, and the remaining terms correspond, respectively, to the realized supplier, plant, warehouse, physical arc, and policy-restricted cross-border capacities. The first-stage decisions $(\bar y,\bar v)$ are fixed in this problem, whereas $(a,b,d,z,h)$ are components of the uncertainty realization selected through $\mathcal{U}^{F}$.
%
%
%
The dual-feasibility constraints are obtained by collecting the coefficients of each nonnegative shipment variable in the recourse Lagrangian. Because the trade-policy terms in the recourse objective apply only to international arcs, the resulting constraints are distinguished according to the arc partition $(E^{D},E^{I})$ defined  in~\eqref{eq:international_domestic_link}. For notational compactness, we define the realized unit cost on arc $(i,j)\in E$ as:
\begin{equation}
\widetilde{c}_{ij}(\xi)=
\begin{cases}
c_{ij}, & (i,j)\in E^{D},\\[1mm]
c_{ij}+\bigl(\underline{\tau}_{ij}
+\Delta\tau_{ij}q_{ij}\bigr)V_{ij},
& (i,j)\in E^{I}.
\end{cases}
\label{eq:realized_arc_cost}
\end{equation}
For notational convenience, extend $\beta$ to all $(i,j)\in E$ by defining
$\beta_{ij}=0$ for every domestic arc $(i,j)\in E^{D}$.
Using~\eqref{eq:realized_arc_cost} and the convention
$\beta_{ij}=0$ for $(i,j)\in E^{D}$, the dual-feasibility conditions for
the three transportation echelons can be written compactly as \eqref{eq:dual_sp}, \eqref{eq:dual_pw}, and \eqref{eq:dual_wr}:
\begin{align}
-\phi_s+\rho_p-\alpha_{sp}-\beta_{sp}
&\leq \widetilde{c}_{sp}(\xi),
&& (s,p)\in E^{SP},
\label{eq:dual_sp}\\
-\gamma_p-\rho_p+\psi_w-\alpha_{pw}-\beta_{pw}
&\leq \widetilde{c}_{pw}(\xi),
&& (p,w)\in E^{PW},
\label{eq:dual_pw}\\
-\theta_w-\psi_w+\lambda_r-\alpha_{wr}-\beta_{wr}
&\leq \widetilde{c}_{wr}(\xi),
&& (w,r)\in E^{WR},
\label{eq:dual_wr}
\end{align}
Constraints~\eqref{eq:dual_sp}--\eqref{eq:dual_wr} correspond to the
shipment variables on the supplier--plant, plant--warehouse, and
warehouse--retailer echelons, respectively. For a domestic arc,
$\beta_{ij}=0$ and $\widetilde{c}_{ij}(\xi)=c_{ij}$. For an international
arc, the corresponding dual constraint additionally contains the
cross-border capacity multiplier $\beta_{ij}$, while
$\widetilde{c}_{ij}(\xi)$ includes the realized tariff component defined
in~\eqref{eq:realized_arc_cost}.
Finally, the coefficient of $\ell_r$ in the recourse Lagrangian is
$\pi_r-\lambda_r$. Since $\ell_r\geq0$, finiteness of the Lagrangian
infimum requires
\begin{equation}
\lambda_r\leq\pi_r,
\qquad r\in R.
\label{eq:dual_shortage}
\end{equation}
Thus, the adversarial problem (AP) associated with the fixed first-stage configuration $(\bar y,\bar v)$ can be written compactly as
\begin{equation}
\label{eq:adversarial_problem}
\mathcal{Q}(\bar y,\bar v)
=
\max_{\substack{
\xi,\,
\phi,\gamma,\theta,\alpha,\beta,\rho,\psi,\lambda}}
\left\{
\mathbb{F}
(\xi,\phi,\gamma,\theta,\alpha,\beta,\rho,\psi,\lambda)
\;\middle|\;
\begin{array}{l}
\xi\in\mathcal{U}^{F}, \eqref{eq:dual_sp}-\eqref{eq:dual_shortage}\\
\phi,\gamma,\theta,\alpha \geq 0\\
\{\beta_{ij}\}_{(i,j)\in E^{I}} \geq0\\
\{\beta_{ij}\}_{(i,j)\in E^{D}} = 0
\end{array}
\right\},
\end{equation}
where $\mathbb{F}(\cdot)$ denotes the objective function in
\eqref{eq:adversarial_objective}. The adversarial problem~\eqref{eq:adversarial_problem} is a mixed-integer
bilinear optimization problem because its objective contains products between
binary uncertainty variables and continuous dual variables, including
$a_s\phi_s$, $b_p\gamma_p$, $d_w\theta_w$, $z_{ij}\alpha_{ij}$, and
$h_{ij}\beta_{ij}$. %An exact linear reformulation of these products through big-$M$ or McCormick constraints would require valid finite bounds on the corresponding dual variables. Since such bounds are not imposed in the proposed formulation, 
We retain the bilinear structure and solve the
adversarial problem directly using a global nonconvex optimization method.

%%%%%%%%%%%%%%%%%%%%%%%%%%%%%%%%%%%%%%%%%%%%%%%%%%%%%%%%%%%%%%%%%%%%%%%%%%%%%%%%
\subsection{Column-and-constraint generation}
\label{sec:algorithm}
%%%%%%%%%%%%%%%%%%%%%%%%%%%%%%%%%%%%%%%%%%%%%%%%%%%%%%%%%%%%%%%%%%%%%%%%%%%%%%%%

The adversarial problem~\eqref{eq:adversarial_problem} evaluates the
worst-case recourse cost associated with a fixed first-stage configuration.
To solve the two-stage robust problem~\eqref{eq:robust_model} without
explicitly incorporating all realizations in $\mathcal{U}^{F}$, we employ a
column-and-constraint generation (CCG) procedure. The method iterates between
a restricted master problem (RMP), which optimizes the first-stage decisions over a
finite set of previously generated uncertainty realizations, and the
adversarial problem, which identifies a worst-case realization for the
resulting first-stage configuration. 
A CCG iteration consists of solving the RMP for the current set of generated
uncertainty realizations, followed by solving the adversarial problem for the
resulting first-stage configuration. We index the iterations by
$k=1,2,\ldots$. 
At each nonterminating iteration, the worst-case realization identified by
the adversarial problem is incorporated into the RMP, together with a
scenario-specific copy of the recourse variables and the corresponding
recourse constraints.

\paragraph{Restricted master problem.}
Let
\begin{equation}
\widehat{\mathcal{U}}^{\,k}
=
\{\xi^{1},\ldots,\xi^{k}\}
\subseteq
\mathcal{U}^{F}
\label{eq:generated_scenario_set}
\end{equation}
denote the set of uncertainty realizations included in the restricted master
problem at iteration $k$. For each
$\xi^{n}\in\widehat{\mathcal{U}}^{\,k}$, let
$(x^{n},\ell^{n})$ denote a scenario-specific copy of the recourse variables,
and let $\eta$ be an epigraph variable for the largest recourse cost among
the generated realizations. The restricted master problem is then obtained as follows:
\begin{align}
Z_{\mathrm{RMP}}^{k}
=
\min_{y,v,\eta,\{x^{n},\ell^{n}\}_{n=1}^{k}}
\quad&
C^{R}(y,v)+\eta
\label{eq:rmp_obj}\\
\text{s.t.}\quad
&
(y,v)\in\mathcal{Y},
\label{eq:rmp_firststage}\\
&
(x^{n},\ell^{n})
\in
\mathcal{F}(y,v,\xi^{n}),
&& n=1,\ldots,k,
\label{eq:rmp_recourse}\\
&
\eta
\geq
C^{O}(x^{n},\ell^{n};\xi^{n}),
&& n=1,\ldots,k.
\label{eq:rmp_eta}
\end{align}
Here, $(y,v)\in \mathcal{Y}$ in \eqref{eq:rmp_firststage} denotes the feasible domain of the first-stage strategic decisions as defined in \eqref{eq:first_stage_domain}; $\mathcal{F}(y,v,\xi^{n})$ denotes the operational feasible region defined by Constraints~\eqref{eq:supplier_capacity}--
\eqref{eq:shortage_bounds}, evaluated at the fixed realization $\xi^{n}$,
and $C^{O}(\cdot)$ is the second-stage operational cost defined in
\eqref{eq:operating_cost}. Constraint~\eqref{eq:rmp_recourse} enforces a
feasible recourse response to each realization in
$\widehat{\mathcal{U}}^{\,k}$, while Constraint~\eqref{eq:rmp_eta} requires
$\eta$ to dominate the corresponding recourse costs. Since $\eta$ is
minimized through~\eqref{eq:rmp_obj}, its optimal value equals the largest
recourse cost among the realizations currently included in the restricted
master problem.
Because $\widehat{\mathcal{U}}^{\,k}\subseteq\mathcal{U}^{F}$, the restricted master
problem optimizes against a subset of the uncertainty realizations considered
in~\eqref{eq:robust_model}. Consequently,
\begin{equation}
LB_k
=
Z_{\mathrm{RMP}}^{k}
\leq
Z^{F},
\label{eq:ccg_lower_bound}
\end{equation}
so $Z_{\mathrm{RMP}}^{k}$ provides a valid lower bound on the optimal value
of the full robust problem. Let $(\bar y^{k},\bar v^{k})$ denote an optimal first-stage configuration obtained from the restricted master problem at iteration $k$.

\paragraph{Adversarial subproblem, scenario generation, and bound update.}

Let $\bar\eta^{\,k}$ denote the optimal value of the epigraph variable
$\eta$ obtained from the RMP at iteration $k$. By
Constraints~\eqref{eq:rmp_recourse}--\eqref{eq:rmp_eta},
$\bar\eta^{\,k}$ represents the largest optimal recourse cost over the
realizations currently contained in
$\widehat{\mathcal{U}}^{\,k}$; equivalently,
$\bar\eta^{\,k}
=
\max_{\xi\in\widehat{\mathcal{U}}^{\,k}}
Q(\bar y^{k},\bar v^{k},\xi)$.
The restricted master problem, however, does not account for realizations in
$\mathcal{U}^{F}\setminus\widehat{\mathcal{U}}^{\,k}$. The candidate
configuration $(\bar y^{k},\bar v^{k})$ is therefore evaluated against the
complete uncertainty set by solving the adversarial
problem~\eqref{eq:adversarial_problem}. Using the worst-case recourse
function defined in~\eqref{eq:worst_case_recourse}, this yields
$\mathcal{Q}(\bar y^{k},\bar v^{k})
=
\max_{\xi\in\mathcal{U}^{F}}
Q(\bar y^{k},\bar v^{k},\xi)$.
Let $\bar\xi^{\,k}\in\mathcal{U}^{F}$ denote a realization attaining this
maximum, so that
$\mathcal{Q}(\bar y^{k},\bar v^{k})
=
Q(\bar y^{k},\bar v^{k},\bar\xi^{\,k})$, where the value on the right-hand
side is obtained through the dual adversarial formulation developed in
Section~\ref{sec:dualization}.
The full robust objective evaluated at the current first-stage configuration
is therefore
\begin{equation}
\widehat Z_k
=
C^{R}(\bar y^{k},\bar v^{k})
+
\mathcal{Q}(\bar y^{k},\bar v^{k}).
\label{eq:ccg_candidate_value}
\end{equation}
By~\eqref{eq:robust_model}, $Z^{F}$ is the minimum of this robust objective
over all feasible first-stage configurations in $\mathcal{Y}$. Since
$(\bar y^{k},\bar v^{k})\in\mathcal{Y}$ is one such feasible
configuration, $Z^{F}\leq\widehat Z_k$. Hence, $\widehat Z_k$ is a valid
upper bound on the optimal robust objective. Maintaining the best such value
obtained through iteration $k$ gives the incumbent upper bound
\begin{equation}
UB_k
=
\min\left\{
UB_{k-1},
\widehat Z_k
\right\},
\qquad
UB_0=+\infty.
\label{eq:ccg_upper_bound}
\end{equation}
Together with the restricted-master lower bound
$LB_k=Z_{\mathrm{RMP}}^{k}$ established in
\eqref{eq:ccg_lower_bound}, the CCG bounds satisfy
$LB_k\leq Z^{F}\leq UB_k$.
The comparison between the adversarial value and the restricted-master
epigraph value determines whether the current RMP omits an uncertainty
realization that increases the worst-case recourse cost of its current
first-stage configuration. From the preceding definitions,
\begin{equation}
\bar\eta^{\,k}
=
\max_{\xi\in\widehat{\mathcal{U}}^{\,k}}
Q(\bar y^{k},\bar v^{k},\xi)
\leq
\max_{\xi\in\mathcal{U}^{F}}
Q(\bar y^{k},\bar v^{k},\xi)
=
\mathcal{Q}(\bar y^{k},\bar v^{k}),
\label{eq:rmp_adversarial_relation}
\end{equation}
where the inequality follows from
$\widehat{\mathcal{U}}^{\,k}\subseteq\mathcal{U}^{F}$.
If the inequality in~\eqref{eq:rmp_adversarial_relation} is strict, the
restricted master problem underestimates the worst-case recourse cost of
$(\bar y^{k},\bar v^{k})$. The realization $\bar\xi^{\,k}$ responsible for
the larger recourse value is therefore added to the generated scenario set:
\begin{equation}
\widehat{\mathcal{U}}^{\,k+1}
=
\widehat{\mathcal{U}}^{\,k}
\cup
\{\bar\xi^{\,k}\}.
\label{eq:scenario_update}
\end{equation}
A scenario-specific copy of the recourse variables and the corresponding
instances of Constraints~\eqref{eq:rmp_recourse} and~\eqref{eq:rmp_eta}
associated with $\bar\xi^{\,k}$ are added to the RMP. The enlarged master problem is then reoptimized, allowing the first-stage configuration to adapt to the newly identified worst-case realization.

In exact arithmetic, the absence of a violated realization is characterized
by $\mathcal{Q}(\bar y^{k},\bar v^{k})=\bar\eta^{\,k}$.
In this case, the restricted master problem already captures the worst-case
recourse value of its current first-stage solution over the complete
uncertainty set. Indeed, using
$LB_k=C^{R}(\bar y^{k},\bar v^{k})+\bar\eta^{\,k}$ and
\eqref{eq:ccg_candidate_value}, this equality implies
$LB_k=\widehat Z_k$. Since $\widehat Z_k$ is the full robust objective value
of the feasible first-stage configuration $(\bar y^{k},\bar v^{k})$, we
have $Z^{F}\leq\widehat Z_k$. Consequently,
$LB_k\leq Z^{F}\leq\widehat Z_k=LB_k$, and hence
$Z^{F}=LB_k=\widehat Z_k$. Therefore,
$(\bar y^{k},\bar v^{k})$ is optimal for~\eqref{eq:robust_model} in the
exact-solution case.
For numerical implementation, termination is based on the lower and
incumbent upper bounds rather than exact equality. At iteration $k$, the relative CCG optimality gap is defined as
\begin{equation}
\operatorname{Gap}_k
=
\frac{UB_k-LB_k}
{\max\{1,|UB_k|\}},
\label{eq:ccg_gap}
\end{equation}
where the denominator normalizes the absolute bound difference by the
magnitude of the incumbent upper bound while preventing division by zero
when $UB_k$ is zero or close to zero. The procedure terminates when
$\operatorname{Gap}_k\leq\varepsilon$, where $\varepsilon>0$ is the
prescribed CCG optimality tolerance.
Otherwise, the newly identified
realization is incorporated through~\eqref{eq:scenario_update}, and the next
CCG iteration is performed. The complete procedure is summarized in
Algorithm~\ref{alg:ccg}.
The CCG procedure is initialized with the nominal realization
$\xi^{1}=\mathbf{0}\in\mathcal{U}^{F}$, for which
$g=a=b=d=z=q=h=0$. This corresponds to the absence of country-level
geopolitical shocks and all associated network-level consequences.
\begin{algorithm}[htbp]
\caption{Column-and-constraint generation (CCG) procedure}
\label{alg:ccg}
\begin{algorithmic}[1]

\State \textbf{Input:} CCG tolerance $\varepsilon>0$

\State \textbf{Initialize:}
$k\gets1$,
$\xi^{1}\gets\mathbf{0}$,
$\widehat{\mathcal{U}}^{\,1}\gets\{\xi^{1}\}$,
$UB_0\gets+\infty$

\State \textbf{Output:}
Incumbent first-stage configuration $(y^{*},v^{*})$, bounds
$LB_k\leq Z^{F}\leq UB_k$ on the optimal objective value of
\eqref{eq:robust_model}, and relative optimality gap
$\operatorname{Gap}_k$

\Procedure{CCG}{}

    \While{true}

        \State
        $\bigl(Z_{\mathrm{RMP}}^{k}, (\bar y^{k},\bar v^{k},\bar\eta^{\,k})\bigr)
        \gets
        \text{Solve }\operatorname{RMP} \text{for } \widehat{\mathcal{U}}^{\,k}$
        using~\eqref{eq:rmp_obj}--\eqref{eq:rmp_eta}
        \label{step:ccg_rmp}
        \vspace{0.3em}
        \State
        $LB_k\gets Z_{\mathrm{RMP}}^{k}$
        \label{step:ccg_lb}
        \vspace{0.3em}
        \State
        $(\mathcal{Q}(\bar y^{k},\bar v^{k}), \bar\xi^{\,k})
        \gets
        \text{Solve adversarial problem (AP) for } (\bar y^{k},\bar v^{k})$
        using \eqref{eq:adversarial_problem}
        \label{step:ccg_ap}
        \vspace{0.3em}
        \State
        $\widehat Z_k
        \gets
        C^{R}(\bar y^{k},\bar v^{k})
        +
        \mathcal{Q}(\bar y^{k},\bar v^{k})$ as in \eqref{eq:ccg_candidate_value}
        \label{step:ccg_candidate}

        \If{$\widehat Z_k<UB_{k-1}$}
            \State $UB_k\gets\widehat Z_k$
            \State $(y^{*},v^{*})\gets(\bar y^{k},\bar v^{k})$
        \Else
            \State $UB_k\gets UB_{k-1}$
        \EndIf
        \label{step:ccg_ub}

        \State
        $\operatorname{Gap}_k
        \gets
        (UB_k-LB_k)/\max\{1,|UB_k|\}$
        \label{step:ccg_gap}

        \If{$\operatorname{Gap}_k\leq\varepsilon$}
            \State \textbf{Return} $(y^{*},v^{*})$, $LB_k$, $UB_k$,
$\operatorname{Gap}_k$
        \EndIf

        \State
        $\widehat{\mathcal{U}}^{\,k+1}
        \gets
        \widehat{\mathcal{U}}^{\,k}
        \cup
        \{\bar\xi^{\,k}\}$
        \label{step:ccg_scenario}

        \State $k\gets k+1$

    \EndWhile

\EndProcedure

\end{algorithmic}
\end{algorithm}
The finite cardinality of $\mathcal{U}^{F}$ yields the following convergence
result.
\begin{theorem}
\label{thm:ccg_convergence}
Suppose that the restricted master problem and the adversarial problem are
solved to global optimality at every iteration. Since $\mathcal{U}^{F}$ is
finite, the CCG procedure terminates after finitely many iterations. In exact
arithmetic and with zero termination tolerance, the resulting first-stage
configuration is optimal for~\eqref{eq:robust_model}.
\end{theorem}

\begin{proof}
All uncertainty variables defining $\mathcal{U}^{F}$ are binary; hence,
$\mathcal{U}^{F}$ contains finitely many admissible realizations. At any
nonterminating iteration, if the worst-case realization $\bar\xi^{\,k}$
identified by the adversarial problem is not contained in
$\widehat{\mathcal{U}}^{\,k}$, it is added according to
Step~\ref{step:ccg_scenario} of Algorithm~\ref{alg:ccg}. Since
$\mathcal{U}^{F}$ is finite, the generated scenario set can be enlarged only
finitely many times.
If instead $\bar\xi^{\,k}\in\widehat{\mathcal{U}}^{\,k}$, then
Constraint~\eqref{eq:rmp_eta} implies $\bar\eta^{\,k}
\geq Q(\bar y^{k},\bar v^{k},\bar\xi^{\,k})
= \mathcal{Q}(\bar y^{k},\bar v^{k})$.
Together with~\eqref{eq:rmp_adversarial_relation}, which gives
$\bar\eta^{\,k}\leq\mathcal{Q}(\bar y^{k},\bar v^{k})$, it follows that $\bar\eta^{\,k}
= \mathcal{Q}(\bar y^{k},\bar v^{k})$. Consequently,
$LB_k=C^{R}(\bar y^{k},\bar v^{k})+\bar\eta^{\,k}
=\widehat Z_k$. Since $LB_k\leq Z^{F}\leq\widehat Z_k$, we obtain
$LB_k=Z^{F}=\widehat Z_k$. Thus, either a previously unrepresented
realization is added to the RMP or the current first-stage configuration is
optimal. Because only finitely many realizations can be added, the CCG
procedure terminates after finitely many iterations, and under exact
arithmetic with zero termination tolerance the returned first-stage
configuration is optimal for~\eqref{eq:robust_model}.
\end{proof}

Theorem~\ref{thm:ccg_convergence} establishes finite convergence under exact
solution of the RMP and adversarial problem. In the computational
implementation, both problems are solved subject to the solver tolerances
specified in Section~\ref{sec:computational}, and the CCG procedure is
terminated when the relative bound gap in~\eqref{eq:ccg_gap} satisfies the
prescribed tolerance $\varepsilon$.

%%%%%%%%%%%%%%%%%%%%%%%%%%%%%%%%%%%%%%%%%%%%%%%%%%%%%%%%%%%%%%%%%%%%%%%%%%%%%%%%
\section{Computational study and model evaluation}
\label{sec:computational}
%%%%%%%%%%%%%%%%%%%%%%%%%%%%%%%%%%%%%%%%%%%%%%%%%%%%%%%%%%%%%%%%%%%%%%%%%%%%%%%%

This section evaluates the proposed robust optimization framework and the CCG
solution procedure developed in Section~\ref{sec:reformulation}. The
computational study has four objectives. First, we assess the computational
performance and scalability of the CCG procedure over a collection of
synthetically generated four-echelon supply chain networks. Second, we examine
the value of explicitly accounting for multiple geopolitical transmission
channels by comparing strategic configurations obtained under the full
uncertainty set $\mathcal{U}^{F}$ and the restricted uncertainty environments
defined in Section~\ref{sec:robust}. Third, we investigate the sensitivity of
the resulting strategic configurations and worst-case operating performance
to the severity and composition of geopolitical uncertainty. Finally, a
case-based analysis evaluates the framework on a supply chain network
constructed from real-world data.

%%%%%%%%%%%%%%%%%%%%%%%%%%%%%%%%%%%%%%%%%%%%%%%%%%%%%%%%%%%%%%%%%%%%%%%%%%%%%%%%
\subsection{Experimental design and implementation}
\label{sec:experimental_design}
%%%%%%%%%%%%%%%%%%%%%%%%%%%%%%%%%%%%%%%%%%%%%%%%%%%%%%%%%%%%%%%%%%%%%%%%%%%%%%%%

The synthetic experiments preserve the four-echelon network structure and
country-aware uncertainty architecture introduced in
Sections~\ref{sec:network} and~\ref{sec:uncertainty}, while systematically
varying network size across the test bed. Six network configurations are
considered, as reported in Table~\ref{tab:instance_classes}. Each network
contains suppliers, manufacturing plants, warehouses, and retailers
distributed across four countries. The number of countries is held fixed
across the six configurations so that changes in computational performance
primarily reflect changes in the underlying supply chain network rather than
simultaneous changes in the dimension of the country-level uncertainty
structure.
For each configuration, ten independently generated instances are considered,
yielding 60 baseline instances in total. Random seeds 2001--2010 are used for
the ten replications of each configuration.
\begin{table}[htbp]
\centering
\caption{Synthetic supply chain network configurations.}
\label{tab:instance_classes}
\begin{tabular}{cccc|cccc}
\toprule
Class & Size group & $|C|$ & ($|S|$, $|P|$, $|W|$, $|R|$) & Class & Size group & $|C|$ & ($|S|$, $|P|$, $|W|$, $|R|$) \\
\midrule
S1 & Small  & 4 & (6, 3, 4, 8) & M2 & Medium & 4 & (9, 4, 6, 13)  \\
S2 & Small  & 4 & (7, 4, 5, 10) & L1 & Large  & 4 & (9, 5, 7, 14) \\
M1 & Medium & 4 & (8, 4, 6, 12) & L2 & Large  & 4 & (10, 5, 7, 15) \\
\bottomrule
\end{tabular}
\end{table}

\textbf{Network generation:} Nodes within each echelon are assigned approximately evenly across the four countries. The resulting country assignments determine the partition of transportation arcs into the international and domestic sets $E^{I}$ and $E^{D}$ defined in~\eqref{eq:international_domestic_link}. Transportation
arcs between consecutive echelons are generated with probabilities $0.40$,
$0.45$, and $0.35$ for $E^{SP}$, $E^{PW}$, and $E^{WR}$, respectively.
The generation procedure additionally ensures that every upstream node has at
least one outgoing arc and every downstream node has at least one incoming
arc.
To maintain comparable demand and capacity conditions as network size
increases, demand and facility capacities are generated relative to aggregate
network demand. Let $D^{\mathrm{tot}}=\sum_{r\in R}D_r$. Total demand is set to
$D^{\mathrm{tot}}=75|R|$.
For each retailer $r\in R$, an independent weight
$\widetilde D_r\sim U[0.75,1.25]$ is generated, and retailer demand is
obtained by normalization: $D_r = D^{\mathrm{tot}}
\frac{\widetilde D_r} {\sum_{r'\in R}\widetilde D_{r'}}$. Thus, average demand per retailer remains constant across network sizes while heterogeneity in retailer demand is preserved.
Facility capacities are generated analogously. Independent weights drawn from
$U[0.75,1.25]$ are normalized within each echelon such that $\sum_{s\in S}U_s=1.50D^{\mathrm{tot}}$, $\sum_{p\in P}U_p=1.35D^{\mathrm{tot}}$, and $\sum_{w\in W}U_w=1.35D^{\mathrm{tot}}$. Accordingly, aggregate capacity-to-demand ratios are held fixed across the six network configurations, while individual facility capacities remain heterogeneous.
Transportation-arc capacities are generated relative to the capacities or
demands associated with their incident nodes. For $(s,p)\in E^{SP}$ and
$(p,w)\in E^{PW}$, $u_{sp} = \vartheta_{sp}\min\{U_s,U_p\}$, $u_{pw} = \vartheta_{pw}\min\{U_p,U_w\}$, where $\vartheta_{sp},\vartheta_{pw}\sim U[0.50,1.00]$. For $(w,r)\in E^{WR}$, $u_{wr}=\vartheta_{wr}D_r$, $\vartheta_{wr}\sim U[1.00,1.50]$.
Because aggregate echelon capacity alone does not ensure that a randomly
generated sparse network can route all retailer demand, each generated
instance is subjected to a nominal serviceability test before being admitted
to the test bed. Specifically, all suppliers are qualified, all generated
transportation arcs are activated, and the nominal uncertainty realization
$\xi=\mathbf{0}$ is imposed. A linear program minimizing aggregate unmet
demand is then solved over the resulting network. An instance is retained only
when zero unmet demand can be achieved within a tolerance of $10^{-6}$;
otherwise, the network topology and transportation-arc capacities are
regenerated. This screening step ensures that service losses observed in the
robust experiments are not caused by an intrinsically incapable nominal
network.

\textbf{Cost and trade-policy parameters:} The parameter-generation rules used in the baseline experiments are reported in Table~\ref{tab:baseline_parameters}. The notation follows
Section~\ref{sec:network}. Parameters are independently sampled from the
specified uniform distributions after an admissible network has been
generated. For each international arc $(i,j)\in E^{I}$, the elevated tariff
rate is obtained as
$\overline{\tau}_{ij}
=\underline{\tau}_{ij}+\Delta\tau_{ij}$, consistent with
\eqref{eq:tariff_equation}.

\begin{table}[htbp]
\centering
\caption{Baseline parameter-generation scheme.}
\label{tab:baseline_parameters}
\begin{tabular}{lll}
\toprule
Parameter & Interpretation & Generation rule \\
\midrule
$\pi_r$ & Unit unmet-demand penalty & $U[40,70]$ \\
$c_{ij}$ & Unit transportation cost & $U[2,10]$ \\
$F_s$ & Supplier qualification cost & $U[100,200]$ \\
$K_{ij}$ & Transportation-arc activation cost & $U[20,60]$ \\
$V_{ij}$ & Customs value per unit, $(i,j)\in E^{I}$ & $U[15,40]$ \\
$\underline{\tau}_{ij}$ & Baseline tariff rate & $U[0.02,0.10]$ \\
$\Delta\tau_{ij}$ & Tariff-rate increment & $U[0.05,0.20]$ \\
$\delta_{ij}$ & Cross-border capacity-reduction fraction & $U[0.20,0.60]$ \\
\bottomrule
\end{tabular}
\end{table}

\textbf{Baseline uncertainty specification:} The baseline experiments use the full multi-channel uncertainty set $\mathcal{U}^{F}$ defined in~\eqref{eq:uncertainty_set}. The country-level
budget is set to $\Gamma_G=1$. The channel-specific budgets in
\eqref{eq:facility_budgets}--\eqref{eq:trade_budgets} are set to $\Gamma_S=\Gamma_P=\Gamma_W
=\Gamma_{SP}=\Gamma_{PW}=\Gamma_{WR}
=\Gamma_T=\Gamma_H=1$. For the aggregate uncertainty budget in~\eqref{eq:joint_budget}, the baseline specification is $\omega_F=2$, $\omega_T=1$, $\omega_H=1$, $\Gamma_J=3$. Under this specification, a physical facility or transportation-arc disruption consumes two units of the joint uncertainty budget, whereas a tariff escalation or policy-induced cross-border capacity restriction consumes one unit. Together with the channel-specific budgets, the joint budget controls both the composition and aggregate severity of the network-level consequences induced by a country-level geopolitical shock.
The baseline specification is used as a controlled stress environment rather
than as an empirical estimate of geopolitical-event probabilities.
Alternative specifications of the aggregate uncertainty budget are examined
subsequently to assess the sensitivity of the results to the assumed
composition and severity of admissible geopolitical realizations.

\textbf{Implementation and performance measures:} The CCG procedure in Algorithm~\ref{alg:ccg} is initialized with the nominal
realization $\xi^{1}=\mathbf{0}$ and terminated when the relative optimality
gap in~\eqref{eq:ccg_gap} satisfies $\operatorname{Gap}_k\leq10^{-5}$. The Gurobi relative MIP optimality gap is set to $10^{-6}$, while the feasibility, integer-feasibility, and optimality tolerances are set to $10^{-7}$. All optimization problems are solved using a single computational thread.
The AP in~\eqref{eq:adversarial_problem} is solved directly in its
binary--continuous bilinear form using Gurobi's global nonconvex optimization
capability. %Consistent with the reformulation in Section~\ref{sec:dualization}, no artificial finite bounds are imposed on the recourse-dual variables for the purpose of big-$M$ linearization. 
At each CCG iteration, the AP objective value is independently checked against the value obtained by solving the primal recourse problem for the uncertainty
realization returned by the AP.

For each instance, we record the final lower and upper bounds $LB_k$ and
$UB_k$, the absolute and relative optimality gaps, the number of CCG
iterations and generated scenarios, the cumulative RMP and AP solution times,
and the total solution time. We additionally record the first-stage
reconfiguration cost $C^{R}(y,v)$, worst-case recourse cost, numbers and
proportions of qualified suppliers and activated transportation arcs, the
composition of the worst-case uncertainty realization, aggregate worst-case
unmet demand, and the corresponding service rate, $1-
\frac{\sum_{r\in R}\ell_r} {\sum_{r\in R}D_r}$. These measures are used in the subsequent analyses to assess computational scalability, strategic network configuration, and operational performance under multi-channel geopolitical uncertainty.

%%%%%%%%%%%%%%%%%%%%%%%%%%%%%%%%%%%%%%%%%%%%%%%%%%%%%%%%%%%%%%%%%%%%%%%%%%%%%%%%
\subsection{Computational performance of the CCG procedure}
\label{sec:ccg_performance}
%%%%%%%%%%%%%%%%%%%%%%%%%%%%%%%%%%%%%%%%%%%%%%%%%%%%%%%%%%%%%%%%%%%%%%%%%%%%%%%%

\textbf{Validation against complete enumeration:}
Before examining computational scalability, we independently validate the CCG
procedure against complete enumeration on a collection of small instances.
Three validation configurations are considered: $V1=(2,3,2,2,3)$, $V2=(3,4,2,3,4)$, $V3=(3,4,3,3,5)$, where each tuple reports the numbers of countries, suppliers, plants, warehouses, and retailers, respectively. Three independently generated instances are considered for each configuration, using random seeds 3001--3003.
For each validation instance, all admissible network-level consequence
realizations $(a,b,d,z,q,h)$ under the baseline uncertainty specification are
enumerated. Different country-shock vectors $g$ that induce the same
network-level consequence realization are represented only once, since $g$
determines the admissibility of the consequences through
\eqref{eq:country_facility}--\eqref{eq:country_tradecapacity} but does not
otherwise enter the operational recourse problem. The complete set of unique
admissible consequence realizations is then incorporated simultaneously into
the master formulation, yielding a deterministic-equivalent robust problem.
Let $Z_{\mathrm{Enu}}$ denote its optimal objective value. The same generated
instance is independently solved using Algorithm~\ref{alg:ccg}, producing
$Z_{\mathrm{CCG}}$.
Table~\ref{tab:enumeration_validation} reports the comparison. The relative
discrepancy is computed as $\Delta_{\mathrm{rel}} = \frac{
|Z_{\mathrm{CCG}}-Z_{\mathrm{Enu}}| }{
\max\{1,|Z_{\mathrm{Enu}}|\} }$.
\begin{table}[htbp]
\centering
\caption{Validation of CCG against complete enumeration.}
\label{tab:enumeration_validation}
\begin{tabular}{cc|cc}
\toprule
Class & S-Enu $\mid$ S-CCG $\mid$ $Z_{\mathrm{Enu}}$ $\mid$  $Z_{\mathrm{CCG}}$ & Class & S-Enu $\mid$ S-CCG $\mid$ $Z_{\mathrm{Enu}}$ $\mid$  $Z_{\mathrm{CCG}}$\\
\midrule
V1 (3001) & 270 $\mid$ 4 $\mid$ 9840.50  $\mid$ 9840.50 & V2 (3001) & 551 $\mid$ 6 $\mid$ 14436.50 $\mid$ 14436.50 
\\
V1 (3002) & 178 $\mid$ 3 $\mid$ 9720.21  $\mid$ 9720.21 & V2 (3002) & 554 $\mid$ 5 $\mid$ 11467.06 $\mid$ 11467.06 
\\
V1 (3003) & 234 $\mid$ 9 $\mid$ 10971.09 $\mid$ 10971.09 &  V2 (3003) & 271 $\mid$ 4 $\mid$ 12712.40 $\mid$ 12712.40 
\\
\hline
V3 (3001) & 668 $\mid$ 5 $\mid$ 15461.49 $\mid$ 15461.49 & V3 (3003) & 839 $\mid$ 9 $\mid$ 13795.84 $\mid$ 13795.84 
\\
V3 (3002) & 651 $\mid$ 8 $\mid$ 18315.69 $\mid$ 18315.69 &  &
\\
\bottomrule
\end{tabular}
\end{table}
The CCG and complete-enumeration objective values agree to numerical precision
for all nine validation instances. %The relative discrepancy is $0$. 
The corresponding CCG lower and upper bounds also
coincide with the complete-enumeration objective values within numerical
precision. These results provide an independent computational verification of
the decomposition procedure on instances for which the complete robust problem
can be explicitly constructed.
The comparison also illustrates the scenario-generation advantage of CCG.
The number of unique admissible consequence realizations ranges from 178 to
839 across the validation instances, whereas CCG requires only 3--9 scenarios
to establish the robust optimum. In aggregate, the nine validation instances
contain 4,216 admissible consequence realizations, while CCG uses only 53
generated scenarios, corresponding to approximately $1.26\%$ of the
enumerated realizations. Thus, the CCG procedure recovers the
complete-enumeration optimum without explicitly incorporating the large
majority of admissible geopolitical realizations.

\textbf{Scalability across the synthetic test bed:} Having validated the CCG procedure against complete enumeration, we next
evaluate its computational performance over the 60 baseline instances
defined in Section~\ref{sec:experimental_design}. All instances satisfy
the prescribed CCG stopping criterion in~\eqref{eq:ccg_gap}.
Table~\ref{tab:ccg_performance} summarizes the computational results
across the six network configurations. Each entry represents an average
over ten independently generated instances unless otherwise indicated.
The table reports network size, CCG iterations, cumulative solution times
for the restricted master problem (RMP) and adversarial problem (AP),
and total solution time. Median and maximum solution times are also
reported to characterize variation across instances.
The results show that computational effort generally increases with
network size, although the relationship is not strictly monotonic.
This suggests that computational difficulty depends not only on network
dimensions but also on characteristics of individual instances,
including network connectivity and the uncertainty realizations
identified during CCG.
The distribution of computational effort between the RMP and AP also
changes with network size. While the AP accounts for a substantial
proportion of solution time in smaller networks, the RMP becomes the
dominant computational component in larger configurations. This is
consistent with Algorithm~\ref{alg:ccg}, where each additional
uncertainty realization introduces new recourse variables and
constraints into the RMP, increasing its size and computational burden.
Despite this growth, the number of CCG iterations remains moderate
across the tested configurations. Together with the complete
enumeration results, these findings indicate that the procedure
obtains solutions satisfying the prescribed numerical tolerance
without explicitly representing all admissible uncertainty
realizations. The increasing computational burden of the RMP also
suggests that improving its solution efficiency may be particularly
valuable for larger networks.
%
\begin{comment}
\paragraph{Scalability across the synthetic test bed.}
Having validated the CCG procedure against complete enumeration, we next
evaluate its computational performance over the 60 baseline instances defined
in Section~\ref{sec:experimental_design}. All 60 instances converge according
to the stopping criterion in~\eqref{eq:ccg_gap}. The largest final relative
optimality gap observed over the complete test bed is
$1.56\times10^{-15}$, which is substantially below the prescribed CCG
tolerance of $10^{-5}$.
%
%
%
%
Table~\ref{tab:ccg_performance} summarizes computational performance by
network configuration. Each entry is based on ten independently generated
instances. In addition to average solution time, the median and maximum are
reported because the larger configurations exhibit substantial
instance-to-instance variation in computational difficulty.
\end{comment}
\begin{table}[htbp]
\centering
\caption{Computational performance of CCG across the synthetic test bed.}
\label{tab:ccg_performance}
\begin{tabular}{lrrrrrrr}
\toprule
Class &
Avg.\ $|E|$ &
Avg.\ iter. &
Avg.\ RMP &
Avg.\ AP &
Avg.\ total &
Median total &
Max total \\
&
&
&
(s) &
(s) &
(s) &
(s) &
(s) \\
\midrule
S1 & 29.0 & 6.6  & 0.22  & 0.61  & 0.83   & 0.82  & 1.29   \\
S2 & 42.1 & 9.4  & 1.67  & 2.43  & 4.10   & 3.93  & 8.78   \\
M1 & 53.0 & 10.5 & 5.92  & 5.16  & 11.08  & 7.24  & 23.94  \\
M2 & 58.1 & 8.6  & 3.36  & 5.15  & 8.51   & 7.98  & 16.53  \\
L1 & 73.0 & 12.1 & 92.22 & 9.56  & 101.79 & 35.49 & 311.54 \\
L2 & 74.8 & 12.5 & 66.75 & 14.30 & 81.05  & 47.22 & 293.37 \\
\bottomrule
\end{tabular}
\end{table}
\subsection{Value of multi-channel geopolitical protection}
\label{sec:multichannel_value}
%%%%%%%%%%%%%%%%%%%%%%%%%%%%%%%%%%%%%%%%%%%%%%%%%%%%%%%%%%%%%%%%%%%%%%%%%%%%%%%%

We next examine the value of incorporating multiple geopolitical transmission
channels when determining the strategic network configuration. The comparison
uses the tariff-only, trade-policy, physical-disruption, and full
multi-channel uncertainty sets defined in
\eqref{eq:tariff_only_set}--\eqref{eq:disruption_only_set} and
\eqref{eq:uncertainty_set}. As established by
Proposition~\ref{prop:nesting}, these uncertainty environments preserve the
same network, first-stage decision structure, and operational recourse model,
but differ in the channels through which country-level shocks may affect the
network.

\textbf{Common full uncertainty stress test:} For each of the 60 baseline instances, we independently optimize the strategic configuration under the four uncertainty environments defined in Section~\ref{sec:robust}. Let $(y^{m},v^{m})$ denote an optimal configuration obtained under environment $m\in\{T,TP,D,F\}$.
As established by Proposition~\ref{prop:nesting}, the optimal robust
objective values depend on the uncertainty sets considered during
strategic design. Comparing these values alone therefore does not
establish how the resulting configurations perform under the same
geopolitical conditions.
To make this comparison, we fix each strategic configuration and
evaluate it under the full uncertainty set $\mathcal{U}^{F}$. The
corresponding total cost is
\begin{equation}
Z_{\mathrm{stress}}^{m}
=
C^{R}(y^{m},v^{m})
+
\max_{\xi\in\mathcal{U}^{F}}
Q(y^{m},v^{m},\xi),
\qquad
m\in\{T,TP,D,F\},
\label{eq:common_stress_cost}
\end{equation}
where $C^{R}(\cdot)$ and $Q(\cdot)$ are defined in
\eqref{eq:config_cost} and \eqref{eq:recourse_obj}, respectively.
This common evaluation isolates the consequences of excluding
particular geopolitical transmission channels during strategic design.
For each restricted environment $m\in\{T,TP,D\}$, the relative
cost penalty is defined as $\operatorname{Penalty}^{m} = 100 \Bigl(\frac{
Z_{\mathrm{stress}}^{m} - Z_{\mathrm{stress}}^{F}}{Z_{\mathrm{stress}}^{F}
}\Bigr)$.
We also measure the difference in service rate relative to the
configuration designed under $\mathcal{U}^{F}$.
Table~\ref{tab:multichannel_stress} summarizes the average results
across the 60 paired instances. The design objective reports the
optimal robust objective under the uncertainty environment used
for strategic design, while the first stage cost represents the
expenditure required to establish the corresponding network
configuration. The full stress total cost measures the combined
strategic and worst case operating costs when that configuration
is exposed to $\mathcal{U}^{F}$. The cost penalty expresses the
relative increase in this total cost compared with the configuration
designed under the full uncertainty set. The service rate measures
the proportion of demand satisfied under the selected worst case
cost realization, while service loss reports the corresponding
difference from the full model configuration in percentage points.
\begin{table}[htbp]
\centering
\caption{Performance of strategic configurations under common
full geopolitical stress.}
\label{tab:multichannel_stress}
\begin{tabular}{lrrrrrr}
\toprule
Design &
Design &
First stage &
Full stress &
Cost &
Service &
Service loss \\
environment &
objective &
cost &
total cost &
penalty (\%) &
rate (\%) &
(pp) \\
\midrule
$\mathcal{U}^{T}$  & 19510.78 & 2265.66 & 29085.83 & 12.06 & 66.91 & 8.82 \\
$\mathcal{U}^{TP}$ & 20474.17 & 2325.87 & 28355.76 &  9.19 & 67.81 & 7.91 \\
$\mathcal{U}^{D}$  & 25011.87 & 2343.63 & 26587.21 &  2.29 & 72.63 & 3.10 \\
$\mathcal{U}^{F}$  & 25982.63 & 2368.25 & 25982.63 &  0.00 & 75.73 & 0.00 \\
\bottomrule
\end{tabular}
\end{table}

The results reveal an important distinction between the cost of
designing against uncertainty and the cost of operating under
geopolitical stress. Restricted uncertainty environments produce
lower design objectives because fewer adverse consequences are
considered. However, these apparent cost advantages do not
persist when the resulting configurations are evaluated under
the full uncertainty set. Although the configuration designed
under $\mathcal{U}^{F}$ has the highest design objective, it
achieves the lowest total cost under common geopolitical stress.
This distinction demonstrates why strategic configurations
developed under different uncertainty assumptions should be
evaluated against a common uncertainty environment.
The comparison also reveals substantial differences in the
protection provided by individual uncertainty channels.
Configurations designed only against tariff escalation exhibit
the greatest cost deterioration. Incorporating restrictions on
international transportation capacity improves performance,
indicating that accounting for changes in trade costs alone
provides limited protection when geopolitical events also
restrict physical flows. Configurations designed against
facility and transportation disruptions perform considerably
closer to those obtained under the full uncertainty set.
This suggests that accounting for physical disruptions provides
broader protection against omitted geopolitical consequences
than accounting for trade policy uncertainty alone, although
the additional channels remain relevant.

\textbf{Cost and service consequences of omitted channels:}
The differences in total cost arise despite relatively small
variations in strategic reconfiguration expenditure. As shown in
Table~\ref{tab:multichannel_stress}, configurations designed under
the full uncertainty set incur slightly higher initial costs but
achieve substantially lower worst case operating costs under the
common stress test. The reduction in operating costs more than
compensates for the additional strategic expenditure. This
indicates that relatively modest investments in network
configuration can provide substantial protection against the
operational consequences of omitted geopolitical uncertainty
channels.
The service results reveal a similar pattern. Configurations
designed under restricted uncertainty environments exhibit lower
average service rates under the common stress test, with the
greatest deterioration observed for designs considering only
tariff escalation. Thus, the consequences of omitting uncertainty
channels extend beyond additional operating costs to reduced
demand fulfillment. As defined in
Section~\ref{sec:experimental_design}, the reported service rates
correspond to the operational responses under the selected
worst case cost realizations.
The differences between restricted uncertainty environments also
suggest an asymmetry in the protection provided by different
channels. Configurations designed against physical disruptions
retain considerably greater protection under full geopolitical
stress than those designed only against trade policy consequences.
One possible explanation is that protection against facility and
transportation failures creates alternative network connections
that can also accommodate changes in trade costs and usable
international transportation capacity. Conversely, configurations
designed only against trade policy uncertainty remain more
exposed when geopolitical shocks also affect the physical
availability of the network. These findings suggest that
protection against physical disruptions provides partial
protection against other geopolitical consequences, although
explicit consideration of trade policy uncertainty remains
valuable.
The strategic configurations provide further insight into this
pattern. Across the 60 instances, the average number of activated
transportation arcs increases from $32.68$ under $\mathcal{U}^{T}$
to $36.68$ under $\mathcal{U}^{F}$, while the corresponding average
number of qualified suppliers decreases slightly from $6.48$ to
$6.25$. The additional protection associated with broader
uncertainty environments is therefore accompanied primarily by
changes in transportation network configuration rather than
increased supplier qualification. This suggests that
transportation redundancy plays an important role in improving
resilience to multiple geopolitical consequences in the tested
networks.

\textbf{Consistency across network configurations:} Table~\ref{tab:multichannel_byclass} reports the average cost
penalties and service losses under the common full uncertainty
stress test for each of the six network configurations. This
comparison examines whether the aggregate findings in
Table~\ref{tab:multichannel_stress} persist across networks
of different sizes.
\begin{table}[htbp]
\centering
\caption{Full-stress performance loss of restricted strategic configurations
by network class.}
\label{tab:multichannel_byclass}
\begin{tabular}{lrrrrrr}
\toprule
&
\multicolumn{2}{c}{$\mathcal{U}^{T}$} &
\multicolumn{2}{c}{$\mathcal{U}^{TP}$} &
\multicolumn{2}{c}{$\mathcal{U}^{D}$} \\
\cmidrule(lr){2-3}
\cmidrule(lr){4-5}
\cmidrule(lr){6-7}
Class &
Cost penalty &
Service loss &
Cost penalty &
Service loss &
Cost penalty &
Service loss \\
&
(\%) & (pp) &
(\%) & (pp) &
(\%) & (pp) \\
\midrule
S1 &  6.65 &  2.20 &  5.13 &  2.50 & 1.78 & 1.51 \\
S2 & 13.75 &  7.61 &  9.38 &  7.29 & 2.03 & 4.02 \\
M1 & 12.03 &  9.55 &  9.44 &  9.32 & 2.12 & 6.19 \\
M2 & 11.87 &  9.48 & 10.37 &  7.94 & 2.81 & 2.28 \\
L1 & 14.75 & 13.78 & 10.29 & 11.68 & 2.30 & 1.40 \\
L2 & 13.31 & 10.30 & 10.51 &  8.73 & 2.70 & 3.19 \\
\bottomrule
\end{tabular}
\end{table}
The results reveal a consistent ordering of cost penalties
across all six network configurations. Designs accounting
only for tariff escalation experience the greatest cost
deterioration, followed by those incorporating both trade
policy channels. Configurations designed against physical
disruptions remain closest to the full model in every network
class. Although the magnitude of these penalties varies,
the ordering is preserved across the synthetic test bed.
Service losses exhibit greater variation and do not always
follow the same ordering as cost penalties. Nevertheless,
all three restricted uncertainty environments exhibit
positive average service losses across the six network
classes. These findings indicate that the consequences of
omitting geopolitical transmission channels extend across
the tested network configurations, although their magnitude
depends on the characteristics of individual networks.

%%%%%%%%%%%%%%%%%%%%%%%%%%%%%%%%%%%%%%%%%%%%%%%%%%%%%%%%%%%%%%%%%%%%%%%%%%%%%%%%
\subsection{Sensitivity to the structure and severity of geopolitical uncertainty}
\label{sec:sensitivity}
%%%%%%%%%%%%%%%%%%%%%%%%%%%%%%%%%%%%%%%%%%%%%%%%%%%%%%%%%%%%%%%%%%%%%%%%%%%%%%%%

We examine how the robust solution responds to changes in the
aggregate uncertainty budget, the relative contributions of
individual consequence channels, and the severity of trade policy
consequences. The analysis considers the M1 and M2 network
configurations in Table~\ref{tab:instance_classes}, using the
same ten random seeds as the baseline experiments. Each parameter
setting is therefore evaluated on 20 common instances, with all
other parameters retained at their baseline values unless
otherwise specified.

\textbf{Aggregate geopolitical severity:} We vary the joint uncertainty budget in~\eqref{eq:joint_budget}
over $\Gamma_J\in\{1,2,3,4,5\}$, retaining the baseline weights
$(\omega_F,\omega_T,\omega_H)=(2,1,1)$.
Table~\ref{tab:gamma_sensitivity} reports the average robust
objective, strategic reconfiguration cost, service rate, and
number of activated transportation arcs. The final three
columns report the average numbers of physical disruptions
($F$), tariff escalations ($T$), and trade capacity restrictions
($H$) in the selected worst case realizations.
\begin{table}[htbp]
\centering
\caption{Sensitivity to the joint uncertainty budget $\Gamma_J$.}
\label{tab:gamma_sensitivity}
\begin{tabular}{rrrrrrrr}
\toprule
$\Gamma_J$ &
$Z^{F}$ &
$C^{R}$ &
Service &
Activated &
$F$ &
$T$ &
$H$ \\
&
&
&
rate (\%) &
arcs &
&
&
\\
\midrule
1 & 20923.24 & 2319.90 & 96.28 & 34.35 & 0.00 & 0.10 & 0.90 \\
2 & 26147.52 & 2397.51 & 80.70 & 36.90 & 1.00 & 0.00 & 0.00 \\
3 & 27143.61 & 2432.60 & 77.28 & 37.70 & 1.00 & 0.10 & 0.90 \\
4 & 30450.84 & 2293.89 & 65.81 & 35.35 & 2.00 & 0.00 & 0.00 \\
5 & 31099.97 & 2327.35 & 63.82 & 36.65 & 2.00 & 0.45 & 0.55 \\
\bottomrule
\end{tabular}
\end{table}
As expected from the expansion of the admissible uncertainty set,
the robust objective increases with $\Gamma_J$. The observed service
rate also declines, indicating greater demand fulfillment losses
under the selected worst case cost realizations. Unlike the robust
objective, however, this service pattern is an empirical observation
rather than a consequence of uncertainty set nesting.
The composition of the worst case realizations illustrates how
the joint uncertainty budget influences the nature of geopolitical
exposure. Under the baseline weights in~\eqref{eq:joint_budget},
physical disruptions are excluded when $\Gamma_J=1$, but become
admissible as the budget increases. The selected realizations
increasingly involve physical disruptions, sometimes accompanied
by trade policy consequences. Thus, increasing the budget changes
not only the range of admissible uncertainty but also the
composition of the adverse scenarios selected by the model.
The strategic response is not monotone. Both the number of
activated transportation arcs and the associated reconfiguration
cost fluctuate as $\Gamma_J$ increases. These findings indicate
that greater geopolitical exposure does not necessarily require
more extensive network activation or higher strategic expenditure,
as changes in admissible disruption combinations can alter the
preferred strategic configuration.

\textbf{Relative composition of geopolitical consequences:}
We examine the sensitivity of the robust solution to the channel
weights in~\eqref{eq:joint_budget}. Holding $\Gamma_J=3$, we consider
four specifications:
$(\omega_F,\omega_T,\omega_H)\in
\{(1,1,1),(2,1,1),(1,2,1),(1,1,2)\}$.
These specifications allow us to examine how the relative budget
requirements of physical disruptions, tariff escalation, and trade
capacity restrictions affect the resulting solutions.
\begin{table}[htbp]
\centering
\caption{Sensitivity to the channel weights in the joint uncertainty budget.}
\label{tab:weight_sensitivity}
\begin{tabular}{rrrrrrrr}
\toprule
$(\omega_F,\omega_T,\omega_H)$ &
$Z^{F}$ &
$C^{R}$ &
Service &
Activated &
$F$ &
$T$ &
$H$ \\
&
&
&
rate (\%) &
arcs &
&
&
\\
\midrule
$(1,1,1)$ & 32610.83 & 2271.75 & 58.92 & 35.95 & 2.95 & 0.05 & 0.00 \\
$(2,1,1)$ & 27143.61 & 2432.60 & 77.28 & 37.70 & 1.00 & 0.10 & 0.90 \\
$(1,2,1)$ & 32609.49 & 2268.03 & 58.92 & 35.85 & 3.00 & 0.00 & 0.00 \\
$(1,1,2)$ & 32603.08 & 2270.07 & 58.92 & 35.90 & 3.00 & 0.00 & 0.00 \\
\bottomrule
\end{tabular}
\end{table}
Table~\ref{tab:weight_sensitivity} shows that the weight assigned
to physical disruptions has a substantial effect on the robust
solution. When $\omega_F=1$, the selected worst case realizations
are dominated by physical disruptions, and the resulting robust
objectives and service rates are nearly identical across the
three corresponding specifications. Increasing the weights
assigned to tariff escalation or trade capacity restrictions
has little effect in these cases because physical disruptions
account for most of the selected adverse consequences.
In contrast, the baseline specification $(2,1,1)$ limits the
number of physical disruptions that can occur within the
same joint budget. The selected worst case realizations
therefore contain fewer physical disruptions and generally
include a trade capacity restriction. This produces a
substantially lower robust objective and higher service
rate than the specifications assigning unit weight to
physical disruptions.
These findings demonstrate that the channel weights influence
not only the admissible combinations of geopolitical
consequences but also the resulting strategic and operational
performance. Their selection is therefore an important
modeling assumption, rather than a simple numerical scaling
choice. The sensitivity results emphasize the need to
interpret the weights as parameters governing the composition
of admissible uncertainty, rather than as empirically
estimated measures of geopolitical severity.

\textbf{Severity of trade policy consequences:} Finally, we examine how the magnitude of tariff escalation and restrictions on international transportation capacity affects the robust solution. Holding the uncertainty budgets and channel
weights at their baseline values, we independently scale the
tariff increments and capacity reduction fractions according to $\Delta\tau_{ij}^{(\kappa_T)} = \kappa_T\Delta\tau_{ij}$, and $\delta_{ij}^{(\kappa_H)}
= \kappa_H\delta_{ij}$, where $\kappa_T,\kappa_H\in\{0.5,1.0,1.5\}$.
The multiplier $1.0$ represents the baseline specification.
All resulting capacity reduction fractions remain within
$[0,1]$ for the parameter ranges considered. Table~\ref{tab:trade_severity} reports the average robust objective,
strategic reconfiguration cost, service rate under the selected
worst case cost realization, and number of activated transportation
arcs across the 20 common instances.
\begin{table}[htbp]
\centering
\caption{Sensitivity to the severity of the trade-policy consequences.}
\label{tab:trade_severity}
\begin{tabular}{llrrrr}
\toprule
Channel varied &
Multiplier &
$Z^{F}$ &
$C^{R}$ &
Service rate (\%) &
Activated arcs \\
\midrule
Tariff increment
& $0.5$ & 27121.62 & 2431.13 & 77.39 & 37.65 \\
& $1.0$ & 27143.61 & 2432.60 & 77.28 & 37.70 \\
& $1.5$ & 27184.68 & 2419.31 & 77.56 & 37.85 \\
\addlinespace
Capacity reduction
& $0.5$ & 26688.54 & 2407.45 & 79.37 & 37.35 \\
& $1.0$ & 27143.61 & 2432.60 & 77.28 & 37.70 \\
& $1.5$ & 27985.58 & 2408.49 & 73.35 & 37.45 \\
\bottomrule
\end{tabular}
\end{table}
Table~\ref{tab:trade_severity} reveals a substantial difference
in the sensitivity of the robust solution to the two trade
policy channels. Changes in tariff increments have a relatively
small effect on the robust objective and service rate over the
tested range. In contrast, increasing the severity of capacity
restrictions produces considerably larger increases in the
robust objective and reductions in service performance.
This difference reflects how the two channels affect operational
decisions. Tariff escalation increases the cost of using
international transportation arcs without directly reducing
their capacity, whereas trade capacity restrictions reduce
the quantity that can be transported through affected arcs.
Consequently, capacity restrictions can limit the network's
ability to satisfy demand, in addition to increasing operating
costs through changes in the available routing alternatives.
The results also show that the strategic response to changes
in trade policy severity is relatively limited under the
baseline uncertainty specification. Both strategic
reconfiguration expenditure and the number of activated
transportation arcs vary only modestly across the tested
multipliers. This suggests that the observed differences in
robust performance arise primarily from changes in the
operational consequences of the uncertainty channels rather
than extensive changes in strategic network configuration.

Overall, the sensitivity analysis demonstrates that the
effects of geopolitical uncertainty depend not only on the
number and composition of admissible consequences but also
on their magnitude. In particular, the stronger response to
capacity restrictions in the tested networks highlights the
importance of representing trade policy interventions through
both cost and capacity consequences, rather than through
tariff escalation alone.

%%%%%%%%%%%%%%%%%%%%%%%%%%%%%%%%%%%%%%%%%%%%%%%%%%%%%%%%%%%%%%%%%%%%%%%%%%%%%%%%
\subsection{Geographically grounded case study}
\label{sec:case_study}
%%%%%%%%%%%%%%%%%%%%%%%%%%%%%%%%%%%%%%%%%%%%%%%%%%%%%%%%%%%%%%%%%%%%%%%%%%%%%%%%

We conclude the computational study with a case study using geographic
and demand information from Ahmadi Digehsara et
al.~\cite{AhmadiDigehsara2025}. Their study uses one year of
international shipment records from the Panjiva database and provides
market locations, demand volumes, and candidate facility locations.
Because the underlying shipment records and a network representation
with four echelons are unavailable, we use this information to
construct a geographically grounded instance of our framework rather
than reproduce their original supply chain network.

\textbf{Case construction and parameterization:} The eleven markets and their reported demand values are shown in Table~\ref{tab:case_demand}. Aggregate demand across these markets is 20,693 units.
\begin{table}[htbp]
\centering
\caption{Market-demand data used in the case-based experiment.}
\label{tab:case_demand}
\begin{tabular}{cc|cc|cc}
\toprule
Market & Demand & Market & Demand & Market & Demand \\
\midrule
Australia  & 929  & US        & 8556  & UK         & 2601 \\
Brazil     & 429  & Hong Kong & 1190  & Turkey    & 125 \\
Canada     & 1907 & India     & 928  & Malaysia   & 1005 \\
Japan      & 1068 & Singapore & 1955 \\
\bottomrule
\end{tabular}
\end{table}
To construct the network defined in Section~\ref{sec:network},
we assign the reported candidate locations to the four
echelons. Bangladesh, Cambodia, India, Pakistan, Sri Lanka,
and Vietnam constitute the supplier set; Hong Kong, Singapore,
Turkey, and Egypt constitute the plant set; and Mexico,
Colombia, and Oman constitute the warehouse set. The eleven
markets in Table~\ref{tab:case_demand} form the retailer set.
Separate nodes are used when a geographic location performs
multiple functions. These functional assignments are specific
to the present study and do not represent classifications
reported in the source data.
The resulting case contains $|S|=6$, $|P|=4$, $|W|=3$, $|R|=11$. Unlike the synthetic test bed in Section~\ref{sec:experimental_design},
transportation connectivity is constructed using geographic proximity rather
than random arc generation. Great-circle distances between locations are computed using the haversine formula \cite{Eason2019}. Transportation arcs
between consecutive echelons are then constructed by connecting
each upstream facility to a prescribed number of its geographically
nearest downstream facilities. Specifically, each supplier is
connected to its three nearest plants, each plant to its three
nearest warehouses, and each warehouse to its seven nearest
retailers. Additional arcs are introduced when necessary to
ensure that every downstream node has at least one incoming
transportation connection. This construction yields a sparse geographically structured network while preserving the sequence $S\rightarrow P\rightarrow W\rightarrow R$.
Table~\ref{tab:case_parameter_sources} summarizes the sources and
construction methods for the case study parameters.
\begin{table}[htbp]
\centering
\caption{Data sources and parameterization of the case study.}
\label{tab:case_parameter_sources}
\begin{tabular}{p{0.30\textwidth}p{0.61\textwidth}}
\toprule
Case component & Source or construction \\
\midrule

Geographic locations &
Locations reported in \cite{AhmadiDigehsara2025}. \\

Market demand &
Demand values for eleven markets from \cite{AhmadiDigehsara2025}. \\

Candidate locations &
Facility locations reported in \cite{AhmadiDigehsara2025}. \\

Four-echelon functional roles &
Constructed for the present study to map the reported locations into the
$S$--$P$--$W$--$R$ architecture. \\

Transportation network &
Obtained via geographic proximity \& nearest location rule. \\

Transportation costs &
Derived from geographic distance and normalized to the cost range used in the
synthetic experiments. \\

Facility capacities &
Calibrated via aggregate demand and baseline capacity ratios. \\

Transportation capacities &
Calibrated via incident facility capacities or retailer demand. \\

Strategic costs &
Midpoints of the baseline ranges in Table~\ref{tab:baseline_parameters}. \\

Unmet demand penalty &
Midpoint of the baseline range in Table~\ref{tab:baseline_parameters}. \\

Trade parameters &
Midpoints of the baseline ranges in Table~\ref{tab:baseline_parameters}. \\

Geopolitical uncertainty &
Full uncertainty set defined in Section~\ref{sec:uncertainty}. \\
\bottomrule
\end{tabular}
\end{table}
Facility capacities follow the aggregate capacity ratios in
Section~\ref{sec:experimental_design}, with total capacities
equal to $1.50D^{\mathrm{tot}}$, $1.35D^{\mathrm{tot}}$, and
$1.35D^{\mathrm{tot}}$ for suppliers, plants, and warehouses,
respectively, where $D^{\mathrm{tot}}=20{,}693$.
Capacity is distributed equally among facilities within each
echelon. Transportation capacities are set to $0.75$ times
the smaller incident facility capacity for supplier to plant
and plant to warehouse arcs, and to $1.25D_r$ for warehouse
to retailer arcs.
For parameters unavailable from the source study, we use
the midpoint values of the corresponding baseline ranges
in Table~\ref{tab:baseline_parameters}. Specifically,
$\pi_r=55$, $F_s=150$, and $K_{ij}=40$. For international
arcs, $V_{ij}=27.5$, $\underline{\tau}_{ij}=0.06$,
$\Delta\tau_{ij}=0.125$, and $\delta_{ij}=0.40$.
The elevated tariff rate therefore satisfies
$\overline{\tau}_{ij}=0.185$ according to
\eqref{eq:tariff_equation}. Transportation costs preserve the geographic
ordering of the network by linearly normalizing great-circle distance to the
$[2,10]$ interval used in the synthetic experiments. %The case therefore combines source-derived geography and demand with transparent calibration of parameters that are unavailable from the source study.
The uncertainty budgets and channel weights retain the
baseline specification from Section~\ref{sec:experimental_design}.

\textbf{Multi-channel design comparison:}
We apply the common-stress evaluation in
Section~\ref{sec:multichannel_value} to the geographically
grounded network. Strategic configurations are optimized
separately under the four uncertainty environments and
subsequently evaluated under the full uncertainty set
$\mathcal{U}^{F}$, according to~\eqref{eq:common_stress_cost}.
Table~\ref{tab:case_results} reports the resulting design
objectives, full-stress costs, cost penalties, and service
performance.

\begin{table}[htbp]
\centering
\caption{Case-based comparison under common full multi-channel stress.}
\label{tab:case_results}
\begin{tabular}{lrrrrr}
\toprule
Design &
Design &
Full-stress &
Cost penalty &
Service &
Service loss \\
environment &
objective &
total cost &
(\%) &
rate (\%) &
(pp) \\
\midrule
$\mathcal{U}^{T}$  & 355740.34 & 702000.10 & 22.69 & 57.13 & 18.33 \\
$\mathcal{U}^{TP}$ & 374647.56 & 611000.49 &  6.79 & 69.70 &  5.77 \\
$\mathcal{U}^{D}$  & 572177.15 & 572177.15 &  0.00 & 75.46 &  0.00 \\
$\mathcal{U}^{F}$  & 572177.15 & 572177.15 &  0.00 & 75.46 &  0.00 \\
\bottomrule
\end{tabular}
\end{table}
The results are consistent with the general findings in
Section~\ref{sec:multichannel_value}. The configuration
designed under $\mathcal{U}^{T}$ exhibits the greatest
deterioration in cost and service performance when evaluated
under the full uncertainty set. Incorporating both trade
policy channels through $\mathcal{U}^{TP}$ substantially
reduces these losses, highlighting the importance of
accounting for restrictions on usable international
transportation capacity in addition to tariff escalation.

A distinctive finding of the geographic case study is that
the configurations designed under $\mathcal{U}^{D}$ and
$\mathcal{U}^{F}$ attain identical optimal robust objective
values and service rates under the common full uncertainty
stress test, despite differing in their activated transportation
arcs. This indicates that, for the present network and
parameterization, accounting for physical disruptions alone
is sufficient to attain the optimal robust objective of
the full model. This case-specific outcome complements the
synthetic results in Section~\ref{sec:multichannel_value},
where configurations designed under $\mathcal{U}^{D}$ incur
a positive average cost penalty under full geopolitical
stress. Together, these findings indicate that the
incremental value of incorporating trade policy consequences
alongside physical disruptions depends on the geographic
and structural characteristics of the supply chain network.

\textbf{Geographic structure of the robust configurations:}
The strategic configurations reveal differences in supplier
qualification and transportation network design. Configurations
developed under $\mathcal{U}^{T}$ and $\mathcal{U}^{TP}$
qualify five suppliers: Bangladesh, Cambodia, India, Pakistan,
and Vietnam. In contrast, configurations developed under
$\mathcal{U}^{D}$ and $\mathcal{U}^{F}$ qualify four suppliers:
Cambodia, India, Pakistan, and Vietnam. Thus, the broader
uncertainty environments do not lead to greater supplier
qualification in this case. Instead, the configurations
differ in their activated transportation corridors, consistent
with the strategic patterns observed in
Section~\ref{sec:multichannel_value}.
The common full uncertainty stress tests identify a
geographically concentrated vulnerability. For all four
strategic configurations, the selected worst case realization
involves a geopolitical shock in Oman, resulting in the
disruption of the Oman warehouse and a restriction on an
inbound international transportation corridor. The affected
corridor is Egypt--Oman for configurations developed under
$\mathcal{U}^{T}$ and $\mathcal{U}^{TP}$, and Hong Kong--Oman
for those developed under $\mathcal{U}^{D}$ and
$\mathcal{U}^{F}$. These consequences are consistent with
the country dependence constraints in
\eqref{eq:country_facility}--\eqref{eq:country_tradecapacity},
which permit multiple adverse consequences to originate
from the same country level shock.
The results identify the Oman warehouse as a critical
location under the selected worst case realizations,
while the affected inbound transportation corridor depends
on the strategic configuration. This illustrates how
country level geopolitical shocks can propagate through
different network components and how strategic decisions
influence the resulting geographic exposure.

%%%%%%%%%%%%%%%%%%%%%%%%%%%%%%%%%%%%%%%%%%%%%%%%%%%%%%%%%%%%%%%%%%%%%%%%%%%%%%%%
\section{Conclusions}
\label{sec:conclusion}
%%%%%%%%%%%%%%%%%%%%%%%%%%%%%%%%%%%%%%%%%%%%%%%%%%%%%%%%%%%%%%%%%%%%%%%%%%%%%%%%

This study develops a robust optimization framework for strategic
supply chain reconfiguration that captures the dependence among
multiple geopolitical consequences arising from common country
level shocks. The proposed uncertainty structure integrates
physical disruptions and trade policy interventions through
channel specific and joint uncertainty budgets, while its
nested restricted sets permit systematic evaluation of
individual transmission channels. The resulting two stage
robust optimization problem is addressed through recourse
dualization and a column and constraint generation procedure.
The computational results show that omitting geopolitical
consequence channels during strategic design can lead to
substantial increases in cost and reductions in service
under common full uncertainty stress. Configurations
designed against physical disruptions generally provide
greater protection against omitted channels than those
considering only trade policy uncertainty. Broader
protection is associated primarily with changes in
transportation network configuration rather than increased
supplier qualification. The sensitivity experiments
demonstrate that both the composition and severity of
geopolitical consequences influence robust performance,
with restrictions on international transportation capacity
having a stronger effect than comparable proportional
changes in tariff increments over the tested ranges.
The geographically grounded case study further shows
that the incremental value of individual channels
depends on network characteristics.

From a managerial perspective, the findings highlight
the importance of evaluating strategic configurations
against the combined consequences of geopolitical shocks.
In particular, alternative transportation connections
can provide operational flexibility without necessarily
requiring additional supplier qualification. The framework
also enables decision makers to identify vulnerabilities
that may be overlooked when geopolitical exposure is
evaluated through isolated trade policy or physical
disruption scenarios.
The analysis considers binary geopolitical consequences
with exogenously specified uncertainty budgets and
channel weights, fixed demand, and linear operating
costs. Future research could develop empirical
calibration methods for geopolitical uncertainty,
incorporate continuous or multiple disruption states,
and extend the framework to multistage settings with
evolving geopolitical conditions and strategic
reconfiguration decisions.

%%%%%%%%%%%%%%%%%%%%%%%%%%%%%%%%%%%%%%%%%%%%%%%%%%%%%%%%%%%%%%%%%%%%%%%%%%%%%%%%
\section*{Declarations}
%%%%%%%%%%%%%%%%%%%%%%%%%%%%%%%%%%%%%%%%%%%%%%%%%%%%%%%%%%%%%%%%%%%%%%%%%%%%%%%%

\noindent
\textbf{Conflict of interest:}
The author declares no conflicts of interest relevant to the
content of this article.

\medskip
\noindent
\textbf{Funding:}
No external funding was received for this research.

\medskip
\noindent
\textbf{Data availability:}
Computational data supporting the reported results are available from the corresponding author upon reasonable request.
\bibliography{sn-bibliography}

\end{document}